\documentclass{scrartcl}

\usepackage[T1]{fontenc}
\usepackage[utf8]{inputenc}
\usepackage[english]{babel}
\usepackage{lmodern}
\usepackage{exscale}
\usepackage[babel]{microtype}
\usepackage{mathtools, mathrsfs}
\usepackage{amssymb}
\usepackage{enumitem}
\usepackage{bbm}
\usepackage{tikz}
\usetikzlibrary{matrix}
\usepackage{tikz-cd}
\usepackage{forest}
\forestset{
  uconf tree/.style={
    for tree={
      grow'=north,
      parent anchor=north,
      child anchor=south,
      l sep=4mm,
      s sep=2mm,
      inner sep=1pt,
      font=\small,
      edge={black},
    },
  },
  rv/.style={red, rectangle, draw=red, inner sep=2pt},
  bv/.style={rectangle, draw=black, inner sep=2pt},
  bx/.style={blue, rectangle, draw=blue, inner sep=2pt},
  lf/.style={inner sep=3pt},
  re/.style={edge={red}},
  de/.style={edge={dashed}},
}

\usepackage{listings}
\usepackage{csquotes}
\usepackage[style=alphabetic, sorting=nyt, maxnames=10, maxalphanames=5, backend=biber]{biblatex}
\DeclareDelimFormat[textcite]{multinamedelim}{\textendash}
\DeclareDelimAlias[textcite]{finalnamedelim}{multinamedelim}
\DeclareNolabel{\nolabel{\regexp{[\p{Z}\p{P}\p{S}\p{C}]+}}}

\usepackage[colorinlistoftodos]{todonotes}
\definecolor{softblue}{RGB}{60, 100, 210}
\DeclareDocumentCommand{\najib}{so+m}{\todo[color=green!70!black, \IfBooleanT{#1}{inline}, caption={\IfValueTF{#2}{#2}{#3}}]{#3}}
\DeclareDocumentCommand{\victor}{so+m}{\todo[color=softblue!80!white, \IfBooleanT{#1}{inline}, caption={\IfValueTF{#2}{#2}{#3}}]{#3}}

\definecolor{corail}{rgb}{0.9882,0.4627,0.4157}
\definecolor{viola}{RGB}{166,146,186}
\definecolor{mocha}{RGB}{164,120,100}
\definecolor{peachfuzz}{rgb}{0.8, 0.5451, 0.3961}
\definecolor{clouddancer}{RGB}{240, 238, 233}
\definecolor{bluefusion}{RGB}{73, 98, 117}
\definecolor{veiledvista2}{RGB}{109, 136, 110}
\usepackage[numbered]{bookmark}
\hypersetup{
  colorlinks,
  linkcolor= veiledvista2 ,
  citecolor= veiledvista2 ,
  urlcolor= veiledvista2 ,
}

\usepackage{amsthm}
\numberwithin{equation}{section}
\theoremstyle{plain}
\newtheorem{theorem}[equation]{Theorem}
\newtheorem{proposition}[equation]{Proposition}
\newtheorem{corollary}[equation]{Corollary}
\newtheorem{lemma}[equation]{Lemma}
\newtheorem{conjecture}[equation]{Conjecture}

\newtheorem{theoremintro}{Theorem}

\newtheorem{conjectureintro}[theoremintro]{Conjecture}

\theoremstyle{definition}
\newtheorem{definition}[equation]{Definition}
\newtheorem{notation}[equation]{Notation}

\theoremstyle{remark}
\newtheorem{example}[equation]{Example}
\newtheorem{remark}[equation]{Remark}

\newcommand{\kk}{\Bbbk}

\newcommand{\F}{\mathbb{F}}
\newcommand{\R}{\mathbb{R}}

\newcommand{\Sym}{\mathbb{S}}
\newcommand{\E}{\mathcal{E}}
\newcommand{\Enu}{\E^{\mathrm{nu}}}
\newcommand{\EE}{\mathbb{E}}
\newcommand{\EEnu}{\EE^{\mathrm{nu}}}
\newcommand{\PP}{\mathcal{P}}
\newcommand{\I}{\mathcal{I}}
\newcommand{\Surj}{\mathcal{X}}
\newcommand{\Surjnu}{\Surj^{\mathrm{nu}}}

\newcommand{\op}{\mathsf{op}}

\newcommand{\dg}{\mathsf{dg}}
\newcommand{\coalg}{\mathsf{coalg}}
\newcommand{\smod}{\mathsf{mod}}
\newcommand{\alg}{\mathsf{alg}}
\newcommand{\Alg}{\mathsf{Alg}}
\newcommand{\Com}{\mathsf{Com}}
\newcommand{\uCom}{\mathsf{uCom}}
\newcommand{\Ass}{\mathsf{Ass}}
\newcommand{\uAss}{\mathsf{uAss}}
\newcommand{\Lie}{\mathsf{Lie}}
\newcommand{\FM}{\mathsf{FM}}
\newcommand{\sLie}{\mathsf{sLie}}
\newcommand{\G}{\mathsf{G}}
\newcommand{\sL}{s\mathbb{L}}
\newcommand{\CC}{\mathcal{C}}

\newcommand{\uPois}{\mathsf{uPois}}

\newcommand{\FI}{\mathsf{FI}}
\newcommand{\FS}{\mathsf{FS}}

\newcommand{\comod}{\text{-}\mathsf{comod}}
\newcommand{\drMod}{\text{-}\mathsf{RMod}}
\newcommand{\hotimes}{\otimes_{\mathrm{H}}}
\newcommand{\dotimes}{\otimes_{\mathrm{D}}}
\DeclareMathOperator{\Tw}{Tw}
\DeclareMathOperator{\Detw}{Detw}
\newcommand{\unit}{\mathsf{1}}

\newcommand{\gk}{\mathfrak{g}}

\newcommand{\qi}{\xrightarrow{ \,\smash{\raisebox{-0.65ex}{\ensuremath{\scriptstyle\sim}}}\,}}

\newcommand{\id}{\mathrm{id}}
\newcommand{\pl}{\mathrm{pl}}
\newcommand{\Q}{\mathbb{Q}}

\DeclareMathOperator{\Conf}{Conf}
\DeclareMathOperator{\UConf}{UConf}
\DeclareMathOperator{\TR}{TR}
\DeclareMathOperator{\BAR}{Bar}
\DeclareMathOperator{\indec}{\mathsf{indec}}
\DeclareMathOperator{\B}{B}
\DeclareMathOperator{\OB}{\Omega B}
\DeclareMathOperator{\OBi}{\Omega_{\iota} B_{\iota}}
\NewDocumentCommand{\ch}{st~}{
  \IfBooleanTF{#1}{
    \operatorname{\IfBooleanTF{#2}{\tilde{N}}{N}}^{*}
  }{
    \operatorname{\IfBooleanTF{#2}{\tilde{N}}{N}}_{*}
}}
\NewDocumentCommand{\adjunction}{mmmm}{
  \begin{tikzcd}[column sep=5pc, ampersand replacement=\&]
    #1
    \arrow[r, shift left=1ex, "#3"{name=F}]
    \&
    #2 \arrow[l, shift left=1ex, "#4"{name=U}]
    \arrow[phantom, from=F, to=U, , "\dashv" rotate=-90]
  \end{tikzcd}
}

\begin{document}
\title{Homology of configuration spaces in positive characteristic via point-set constructions}
\author{Najib Idrissi\thanks{Université Paris Cité, Sorbonne Université, CNRS, IMJ-PRG, F-75013 Paris, France} \and Victor Roca i Lucio\thanks{Université Paris Cité, Sorbonne Université, CNRS, IMJ-PRG, F-75013 Paris, France}}
\date{June 25, 2026}
\maketitle

\begin{abstract}
  The first goal of this paper is to provide concrete chain complexes computing the homology of (unordered) configuration spaces of manifolds in positive characteristic, lifting a theorem by Knudsen to the model category level. We make them fully explicit and provide a computer program to compute their homology. Our methods also allow us to construct several new spectral sequences converging to these homology groups. Finally, we conjecture that this equivalence of chain complexes can be promoted to an equivalence of \emph{twisted} $\EE_\infty$-coalgebras in right $\EE_d$-modules, and we explain how this conjecture would imply the homotopy invariance of the $\EE_d$-homotopy type of configuration spaces in positive characteristic via new ``twist'' and ``detwist'' functors.
\end{abstract}

\tableofcontents

\section{Introduction}

Configuration spaces of points are classical objects of study in algebraic topology, with applications in various fields of mathematics and physics.
Given a topological space $M$, its $k$-th ordered configuration space is defined as:
\begin{equation}
  \Conf_k(M) \coloneqq \{(x_1, \dots, x_k) \in M^k \mid x_i \neq x_j \text{ for } i \neq j\}~.
\end{equation}
The symmetric group $\Sym_k$ acts freely on $\Conf_k(M)$ by permuting the points, and the $k$-th unordered configuration space is defined as the quotient $\UConf_k(M) \coloneqq \Conf_k(M)/\Sym_k$.
We refer to~\cite{knudsen_configuration_2018,kallel_configuration_2025} for recent surveys on the topic.

The homology of configuration spaces has been widely studied, although a complete understanding remains elusive in many cases.
Simple questions such as ``what is the $i$-th Betti number of $\UConf_k(M)$?'' often have difficult answers.
In addition, knowledge of the homotopy type of $M$ is not sufficient to determine the homotopy type of its configuration spaces~\cite{longoni_configuration_2005}.
Some results are known in characteristic zero~\cite{kriz_rational_1994,totaro_configuration_1996,Idrissi2019,CamposWillwacher2023,campos_configuration_2024}, but the general case is much less understood (see, e.g., \cite{levitt_spaces_1995} for $k=2$ when $M$ is $2$-connected).

In this article, we provide new tools to study the homology of configuration spaces in positive characteristic based on results of~\cite{Knudsen18}.
In particular, we provide explicit chain complexes (``point-set models'') whose homology computes the homology of labeled configuration spaces of parallelizable manifolds.
These labeled configuration spaces can recover unordered configuration spaces (with trivial labels), ordered configuration spaces, and ``everything in between.''
While some computations are known for $p = 2$ or odd-dimensional manifolds~\cite{BODIGHEIMER1989111}, the general case remains difficult.

The main tools that we use to define these point-set models from Knudsen's results are the recent operadic constructions of~\cite{LeGrignouRocaiLucio2023}, as well as explicit models for the $\EE_\infty$ operad such as the Barratt--Eccles operad $\E$~\cite{McClureSmith2003,BergerFresse04}.
A key tool we use is the fact that the cooperad $\B(\sLie \otimes \E)$ is quasi-planar (Section~\ref{sec:quasi-planar-cooperads}).
Therefore, the bar-cobar construction on $\sLie \otimes \E$ a model for the spectral Lie operad $\sL$.
Moreover, we can use the functorial cofibrant resolutions given by the operadic calculus in \cite{LeGrignouRocaiLucio2023} to give point-set models for the derived indecomposables functor in terms of an operadic bar construction.

Let us say that $M$ is of finite type if its homology is of finite dimension in each degree.

\begin{theoremintro}[{See Theorem~\ref{thm:main}}]\label{thmA}
  Let $\kk$ be an arbitrary field.
  Let $M$ be a parallelizable $d$-manifold of finite type, and let $V$ be a chain complex over $\kk$. There is a zig-zag of quasi-isomorphisms of weighted chain complexes over $\kk$:
  \begin{equation}\label{equation: intro zig-zag}
    \bigoplus_{k \geq 1} \ch(\Conf_k(M)) \otimes_{h\Sym_k} V^{\otimes k} \simeq \B_\iota \bigl( \ch*~(M^+) \otimes \OB(\sLie \otimes \E) \circ (s^d V)\bigr)~,
  \end{equation}
  where $\ch*~(M^+)$ is the reduced normalized cochain complex of the one-point compactification of $M$, viewed as a nonunital $\E$-algebra; $\OB(\sLie \otimes \E) \circ (s^d V)$ is a model for the free spectral Lie algebra on the $d$-fold suspension of $V$; $\B_\iota$ is the operadic bar construction; on the left-hand side, the weight is given by $k$, and on the right-hand side, the weight is given by considering $V$ to be in weight $1$.
\end{theoremintro}

Using factorization homology, Knudsen obtained rational Lie models for unordered configuration spaces~\cite{Knudsen18,Knudsen17}.
These models were used to reprove and refine homological stability~\cite{Knudsen17} and enabled the computation of Betti numbers for surfaces~\cite{DrummondColeKnudsen17}.
To address positive characteristic, subsequent works developed spectral sequences relying heavily on the power operations of spectral Lie algebras~\cite{BrantnerKnudsenHahn2019,zhangQuillenHomologySpectral2025}.
In particular, \textcite{ChenZhang2022} applied a topological André--Quillen (TAQ) spectral sequence (described by \textcite{zhangQuillenHomologySpectral2025}) to prove that configuration spaces of surfaces have no $p$-torsion, recovering previous results obtained via $E(n)$-local computations in \textcite{BrantnerKnudsenHahn2019}.
These approaches use the bar spectral sequence whose $E_2$ page is determined by the homotopy groups of $\sL(\Sigma^d X)^{M^{+}}$.
They thus depend on knowledge of these homotopy groups of $\sL(\Sigma^d X)$, as well as on knowledge of power operations on spectral Lie algebras (see~\cite{antolin-camarena_mod_2020,kjaer_odd_2018}).

Since our point-set models are explicit, we can leverage their concreteness to perform computations via new methods.
For instance, we can consider the natural filtration of the Barratt--Eccles operad by complexity~\cite{BergerFresse04}, which gives models for the $\mathbb{E}_n$ operads as suboperads of the $\mathbb{E}_\infty$ operad, and obtain a new filtration and spectral sequence that converges to the homology of these configuration spaces. We can also vary the point-set models and obtain smaller complexes by replacing the Barratt--Eccles model with the surjection operad model of \cite{McClureSmith2003,BergerFresse04}. Moreover, we develop a computer program that encodes these complexes and allows us to carry out explicit computations in the case where $M = \mathbb{R}^n$. Let us stress an important point: the only ``unknown'' aspect, for a general $M$, that is not fully determined in our approach is $\ch*~(M^+)$ together with its $\mathbb{E}_\infty$-coalgebra structure. While we do compute it explicitly for $M = \mathbb{R}^n$, this is precisely, in our opinion, the key element needed to carry the approach and the computations of the present paper further, and it should be the subject of future work.

\medskip

One of the main interests of Theorem~\ref{thmA}, for us, is its relation to the study of the $p$-adic homotopy type of configuration spaces and, in particular, to whether it is an invariant of the $p$-adic homotopy type of the manifold $M$ when $M$ is simply connected. By the work of Mandell in \textcite{Mandell2001}, we know that (finite-type) nilpotent $p$-adic homotopy types can be fully faithfully encoded by $\EE_\infty$-algebras over an algebraically closed field of characteristic $p$. Moreover, one can remove the finite-type assumption if one works with $\EE_\infty$-coalgebras, using the results of \cite{Bachmann2024}.

\medskip

Both sides in Theorem~\ref{thmA} have the flavor of an $\EE_\infty$-coalgebra; in particular, for every $k \geq 0$, $\ch(\Conf_k(M))$ is naturally an $\EE_\infty$-coalgebra, and it encodes the $p$-adic homotopy type of $\Conf_k(M)$. However, on the right-hand side, it is not clear how to make this object live in a category of $\EE_\infty$-coalgebras.

\medskip

Our second goal in this paper is to leverage Theorem~\ref{thmA} to construct an explicit framework in which one can make precise conjectures related to the homotopy invariance of configuration spaces in positive characteristic. First, we observe that while the right-hand side of Theorem~\ref{thmA} is not a priori an $\EE_\infty$-coalgebra at every weight $k \geq 0$, the whole collection can be lifted to a symmetric sequence endowed with a \emph{twisted} $\EE_\infty$-coalgebra structure. Then, we build what we call the ``twist'' and ``detwist'' functors, which relate the monoidal structures on symmetric sequences that we are interested in, namely, the arity-wise tensor product and the Day convolution tensor product. Using the underlying action of the $\EE_d$-operad on the symmetric sequence $\ch(\Conf_k(M))$, we twist its $\EE_\infty$-coalgebra structure, which allows us to conjecture the following.

\begin{conjectureintro}[Conjectures~\ref{conj:e-infty-coalg} and \ref{conj:e-infty-coalg strong}]
  Let $M$ be a simply connected parallelizable manifold of finite type and dimension $d$. The zig-zag of quasi-isomorphisms in \eqref{eq:conj-model} can be promoted to a zig-zag of equivalences
  \[
    \Tw(\ch(\mathsf{FM}_M)) \simeq \B_\iota \bigl( \ch*~(M^+) \otimes \OB(\sLie \otimes \E) \circ (s^d \I)\bigr),
  \]
  of twisted $\EE_\infty$-coalgebras in right $\EE_d$-modules.
\end{conjectureintro}

We first explain how this conjecture holds in the rational case and how it implies the homotopy invariance results of \cite{Idrissi2019,CamposWillwacher2023}. Finally, we explain what types of consequences the $p$-adic version would have for the homotopy invariance of configuration spaces: namely, while it would not imply full homotopy invariance (due to the lack of a fully symmetric detwist functor, as explained in Remark~\ref{rmk: formality implies fully symmetric detwist}), it would still imply the invariance of the homotopy type of $\ch(\Conf_k(M))$ as an $\EE_d$-coalgebra. Our hope is that this framework, combined with the very explicit objects constructed in this paper, can prove useful to tackle the many open problems that remain about the $p$-adic homotopy theory of configuration spaces.

\subsection{Conventions and notation}

In what follows, $\kk$ denotes a fixed field.
We take all chain complexes to be over $\kk$.
Given a list of variables $x_1, \dots, x_n$, we denote by $\kk\langle x_1, \dots, x_n \rangle$ the $n$-dimensional vector space freely generated by these elements.
We also denote linear duality by $(-)^\vee = \hom(-, \kk)$.

Given a simplicial set $X$, we let $\ch(X)$ be the chain complex of normalized (cellular) chains of $X$ over $\kk$, and $\ch*(X)$ the dual cochain complex.
When $X$ is pointed, we moreover set $\ch*~(X) \coloneqq \ker\bigl(\ch*(X) \to \kk\bigr)$ and $\ch~(X) \coloneqq \operatorname{coker}\bigl(\kk \to \ch(X)\bigr)$.

Throughout the paper, we let $M$ denote a parallelizable manifold of dimension $d$, and we let $X$ be any spectrum.
The unordered configuration spectrum of $k$ points labeled by $X$ is defined as
\[
  \B_k(M;X) \coloneqq \Sigma_{+}^{\infty} \Conf_k(M) \otimes_{h\Sym_k} X^{\otimes k}~,
\]
where $\Conf_k(M) \coloneqq \{ \underline{x} \in M^k \mid i \neq j \implies x_i \neq x_j \}$ denotes the ordered configuration space of $k$ points in $M$, $\Sigma_{+}^{\infty}$ the infinite suspension functor, and $\otimes$ the smash product of spectra.
We also let $\UConf_k(M) \coloneqq \Conf_k(M)/\Sym_k$ denote the unordered configuration space of $k$ points in $M$.

\subsection{Acknowledgements}

The authors acknowledge support from project ANR-22-CE40-0008 SHoCoS, ANR-20-CE40-0016 HighAGT, and the IdEx University of Paris ANR-18-IDEX-0001.
N.\ I.\ acknowledges support from the Institut Universitaire de France (IUF).
We thank Adela Zhang and Connor Malin for useful discussion.

The software package ComCH~\cite{medina-mardones_comch_nodate} was used for some computer calculations.
During the preparation of this work, the authors used Claude, Gemini, and GPT LLMs to assist with writing and debugging the source code used for the numerical computations. On the contrary, the math and the writing of the paper were done by humans and no LLMs were involved.
The authors accept full responsibility for the accuracy and correctness of the resulting code and data.

\section{A point-set model for unordered configuration spaces in positive characteristic}

In this section, we lift Knudsen's theorem~\cite{Knudsen18} about the stable homotopy type of unordered configuration spaces of parallelizable manifolds to the level of model categories when working over a field. This allows us to provide an explicit chain complex whose homology is the homology of these unordered configuration spaces, possibly labeled by a chain complex.

\subsection{Point-set models for spectral Lie algebras}
We start by choosing a point-set model for the $\infty$-category of spectral Lie algebras in chain complexes over a field $\kk$ that will be convenient for our purposes. Spectral Lie algebras can be defined as algebras over the enriched spectral Lie $\infty$-operad $\sL$. This notion was first introduced in~\cite{salvatore_configuration_1998,ching_bar_2005}; for an $\infty$-categorical treatment, we refer, for instance, to \cite[Section 5]{HeutsSurvey}.

Let $\sLie$ denote the shifted Lie operad over $\kk$, generated by a single arity $2$ operation of degree $-1$, which corresponds to the Lie bracket.
Let $\E$ denote the Barratt--Eccles operad~\cite{BergerFresse04} and $\Enu$ its non-unital version, where $\Enu(0) \cong 0$.
The Hadamard tensor product $\sLie \otimes \E$ gives an $\Sym$-projective resolution of the operad $\sLie$.
Applying the operadic bar-cobar functors yields a dg operad $\OB(\sLie \otimes \E)$ which is a cofibrant resolution of $\sLie$ in the semi-model structure of dg operads.
The category of dg $\OB(\sLie \otimes \E)$-algebras thus admits a transferred model structure from the underlying category of chain complexes over $\kk$, where weak equivalences are given by quasi-isomorphisms and fibrations by degree-wise surjections.

\begin{notation}
  Note that $\sLie(0) = 0$; thus $\sLie \otimes \E = \sLie \otimes \Enu$ and both operads are augmented. For simplicity, we will continue writing $\OB(\sLie \otimes \E)$ instead of $\OB(\sLie \otimes \Enu)$.
  While they are equal, the former is more convenient, even though the latter is more precise.
\end{notation}

\begin{proposition}\label{prop:rectification of spectral Lie algebras}
  There is an equivalence of $\infty$-categories
  \[
    \dg\text{-}\OB(\sLie \otimes \E)\text{-}\alg~[\mathrm{q.iso}^{-1}] \simeq \Alg_{\sL}(\smod_\kk)~,
  \]
  between the $\infty$-category of $\OB(\sLie \otimes \E)$-algebras localized at quasi-isomorphisms and the $\infty$-category of spectral Lie algebras in $\kk$-modules.
\end{proposition}

\begin{proof}
  This follows from \cite[Theorem 4.10]{Hauseng19}, since $\OB(\sLie \otimes \E)$ is both admissible and $\Sym$-cofibrant (in fact, $\Sym$-free); it is therefore flat in the terminology of \emph{op.cit.} since the base category of chain complexes over a field $\kk$ is trivially a category equipped with a subcategory of flat objects.
\end{proof}

\begin{remark}
  Any admissible $\Sym$-projective resolution of the operad $\sLie$ would work in the above proposition.
  We choose this particular resolution to be able to apply the results of~\cite{LeGrignouRocaiLucio2023}, as the cooperad $\B(\sLie \otimes \E)$ is \emph{quasi-planar} (Section~\ref{sec:quasi-planar-cooperads}).
\end{remark}

\subsection{Point-set models for the \texorpdfstring{$\infty$}{infinity}-categorical derived indecomposables functor}

We now provide a point-set model for the $\infty$-categorical derived indecomposables functor. As a byproduct, we obtain an explicit chain complex that models the $\infty$-categorical bar construction appearing in Knudsen's theorem.

\medskip

We begin by recalling the $\infty$-categorical construction of the derived indecomposables functor in the particular case of spectral Lie algebras. We refer to \cite[Section 3]{FrancisGaitsgory12} or to \cite[Section 4]{HeutsSurvey} for more details. The spectral Lie operad $\sL$ is canonically augmented, as it is trivial in arity $1$. Its augmentation $\sL \longrightarrow \mathrm{I}$ induces an $\infty$-categorical adjunction
\[
  \adjunction{\Alg_{\sL}(\smod_\kk)}{\smod_\kk}{\indec}{\mathsf{triv}}
\]
between the base $\infty$-category of chain complexes over $\kk$ up to quasi-isomorphisms and the $\infty$-category of spectral Lie algebras over $\kk$. The right adjoint $\mathsf{triv}$ endows a chain complex $V$ with a trivial spectral Lie algebra structure. The left adjoint $\indec$ is called the derived indecomposable functor. Its homology is often called the topological André--Quillen (TAQ) homology of spectral Lie algebras; see \cite{zhangQuillenHomologySpectral2025}. Over a field of characteristic zero, this homology is computed by the Chevalley--Eilenberg complex; in general, it is more complicated. Our goal here is to provide a point-set model for this derived indecomposable functor using the operadic constructions of \cite{LeGrignouRocaiLucio2023}. In particular, we have the following.

\begin{theorem}[{\cite[Section 3.4, Proof of Theorem C]{Knudsen18}}]\label{thm:knudsen}
  There is a weak equivalence of weighted spectra
  \[
    \bigoplus_{k \geq 1} \B_k(M;X) \simeq \biggl| \BAR_{\bullet}\Bigl(\mathrm{id}, \sL, \sL(\Sigma^d X)^{M^{+}}\Bigr) \biggr|~,
  \]
  where $(-)^{M^+}$ denotes the cotensoring by the one-point compactification of $M$ and where, on the right-hand side, the weight is given by $k \geq 1$, while on the left-hand side, the weight is given by considering $X$ to be in weight $1$.
\end{theorem}

The right-hand side in the above theorem can also be described as the derived indecomposables functor applied to the spectral Lie algebra $\sL(\Sigma^d X)^{M^{+}}$:

\begin{corollary}[{\cite{Knudsen18}}]\label{cor:indec-bar}
  Let $M$ be a parallelizable manifold of dimension $d$. There is a weak equivalence of weighted spectra
  \[
    \bigoplus_{k \geq 1} \B_k(M;X) \simeq \indec\Bigl(\sL(\Sigma^d X)^{M^{+}}\Bigr)~,
  \]
  where $\sL(\Sigma^d X)$ denotes the free spectral Lie algebra on $\Sigma^d X$, $(-)^{M^+}$ denotes the cotensoring by the one-point compactification of $M$ (see Section~\ref{sec:infty-categorical-cotensor}); and where on the right-hand side, the weight is given by $k \geq 1$, while on the left-hand side, the weight is given by considering $X$ to be in weight $1$.
\end{corollary}

\begin{proof}
  This follows from Theorem~\ref{thm:knudsen} and the following general fact.
  Since the $\infty$-category of $\sL$-algebras is monadic, any weighted $\sL$-algebra $A$ is resolved by the geometric realization of the two-sided simplicial bar construction
  \begin{align*}
    A & \simeq \bigl|\BAR_\bullet(\sL,\sL,A) \bigr|,
    & \BAR_\bullet(\sL,\sL,A)                      & \coloneqq \sL^{\circ \bullet} \circ A,
  \end{align*}
  where the weight is induced by the weight of $A$.
  Since the (derived) indecomposables functor $\indec$ commutes with homotopy colimits, we have:
  \[
    \mathbb{L}\indec(A) \simeq \Bigl|\indec\bigl(\BAR_\bullet(\sL,\sL,A)\bigr) \Bigr| \simeq \bigl| \BAR_\bullet(\mathrm{id},\sL,A) \bigr|~,
  \]
  where $\mathrm{id}$ is the identity endofunctor of the underlying $\infty$-category.
\end{proof}

Let $\epsilon: \OB(\sLie \otimes \E) \longrightarrow \mathrm{I}$ be the augmentation of our point-set model for the spectral Lie operad. It induces a Quillen adjunction
\[ \adjunction{\dg~\OB(\sLie \otimes \E)\text{-}\alg}{\dg~\smod}{\mathrm{indec}}{\mathrm{triv}}, \]
since chain complexes over $\kk$ are exactly algebras over the trivial operad $\mathrm{I}$.

\begin{lemma}\label{lemma: derived indecomposables}
  The derived point-set adjunction
  \[ \adjunction{\dg~\OB(\sLie \otimes \E)\text{-}\alg~[\mathrm{q.iso}^{-1}]}{\dg~\smod~[\mathrm{q.iso}^{-1}]}{\mathbb{L}\mathrm{indec}}{\mathrm{triv}} \]
  presents the derived indecomposables adjunction at the $\infty$-categorical level.
\end{lemma}

\begin{proof}
  Follows directly from Proposition \ref{prop:rectification of spectral Lie algebras}, since the right adjoints can be easily identified.
\end{proof}

Recall that the universal twisting morphism $\iota: \B(\sLie \otimes \E) \longrightarrow \OB(\sLie \otimes \E)$ induces a bar-cobar adjunction
\[ \adjunction{\dg~\B(\sLie \otimes \E)\text{-}\mathsf{coalg}}{\dg~\OB(\sLie \otimes \E)\text{-}\alg}{\B_{\iota}}{\Omega_{\iota}}, \]
between the category of dg $\OB(\sLie \otimes \E)$-algebras and the category of dg $\B(\sLie \otimes \E)$-coalgebras.

\begin{proposition}\label{prop: computation of derived indecomposable}
  Let $A$ be a weighted dg $\OB(\sLie \otimes \E)$-algebra.
  The value of the derived indecomposable functor at $A$ can be computed by the chain complex $\B_\iota(A)$, which inherits a weight from $A$.
\end{proposition}

\begin{proof}
  By~\cite[Theorem 7]{LeGrignouRocaiLucio2023}, $A$ admits a functorial cofibrant resolution given by the bar-cobar adjunction relative to $\iota$:
  \[
    \epsilon_A: \OBi A \qi A.
  \]
  Moreover, this morphism preserves the weight.
  So, we can compute that
  \[
    \operatorname{\mathbb{L}indec}(A) \simeq \operatorname{indec}(\OBi A) \simeq \B_{\iota}A. \qedhere
  \]
\end{proof}

\subsection{Point-set models for the cotensoring of spectral Lie algebras over spaces}

The $\infty$-category of spectral Lie algebras is cotensored in spaces. We construct an explicit point-set model for this cotensoring that appears in Knudsen's theorem using the results of \cite{LeGrignouRocaiLucio2023}.

\subsubsection{The \texorpdfstring{$\infty$}{infinity}-categorical cotensor of spectral Lie algebras}\label{sec:infty-categorical-cotensor}

Let $X$ be a space and let $\gk$ be a spectral Lie algebra over $\kk$.
Consider the homotopy limit $\lim_X \gk$, where $\gk$ is seen as a constant diagram over the $\infty$-groupoid $X$~\cite[Section~3.4.3]{Lurie2017}. The underlying spectrum of $\lim_X \gk$ is given by the mapping spectrum $\operatorname{Map}_{\mathrm{Sp}}(\Sigma^{\infty}_+X, \gk)$. Hence, if $X$ is a finite type space, the underlying spectrum of $\lim_X \gk$ is given by $(\Sigma^{\infty}_+X)^* \wedge \gk$.

This construction also admits a pointed version which presents the cotensoring of Lie algebras in pointed spaces.
\[
  \gk^X \coloneqq \operatorname{fib} \bigl( \lim_X \gk \longrightarrow \gk \bigr) ,
\]
where we take the fiber along the map induced by the pointing of $X$. It follows that if $X$ is a finite type pointed space, then the underlying spectrum of $\gk^X$ is given by $(\Sigma^{\infty}X)^* \wedge \gk$. We refer to \cite{BrantnerKnudsenHahn2019}, and in particular to the proof of Proposition 5.8, for more details on these constructions.

\subsubsection{Quasi-planar cooperads}\label{sec:quasi-planar-cooperads}

We will heavily use the results of \cite{LeGrignouRocaiLucio2023} in what follows.
These results hinge on the notion of quasi-planar cooperads, which we now briefly recall.
We will only ever consider non-curved cooperads, so we restrict to this case for simplicity.
In what follows, we let $\bar{\mathbb{T}}$ denote the reduced tree comonad, $\mathrm{Cooperads}^{\mathrm{conil}}$ the category of conilpotent cooperads, and $\mathrm{grCooperad}^{\mathrm{conil}}_{\pl}$ the category of conilpotent graded planar cooperads.
We add the modifier $\mathrm{dg}$ and $\mathrm{gr}$ to indicate differential graded and graded objects respectively.

\begin{definition}[{\cite[Definition~33]{LeGrignouRocaiLucio2023}}]
  An \emph{cooperad ladder} is the data of an ordinal $\alpha$ and a cocontinuous functor $\CC : \alpha \to \mathrm{dgCooperads}^{\mathrm{conil}}, i \mapsto \CC^{(i)}$ such that:
  \begin{enumerate}[nosep]
    \item for every $i \in \alpha$, the map $\CC^{(i)} \to \CC^{(i+1)}$ is injective;
    \item the reduced comultiplication $\Delta : \CC^{(i+1)} \to \bar{\mathbb{T}}\CC^{(i+1)}$ factors through $\bar{\mathbb{T}}\CC^{(i)}$.
  \end{enumerate}
\end{definition}

\begin{definition}[{\cite[Definition~34]{LeGrignouRocaiLucio2023}}]
  A \emph{quasi-planar cooperad ladder} is the data of a small ordinal $\alpha$ and a commutative diagram:
  \begin{equation}
    \begin{tikzcd}
      \alpha \arrow[r, "\CC"] \arrow[d, "\CC_{\mathrm{pl}}"'] & \mathrm{dgCooperads}^{\mathrm{conil}} \arrow[d, "\mathrm{U}"] \\
      \mathrm{grCooperad}^{\mathrm{conil}}_{\pl} \arrow[r, "\Sym \otimes -"] & \mathrm{grCooperads}^{\mathrm{conil}},
    \end{tikzcd}
  \end{equation}
  such that:
  \begin{enumerate}[nosep]
    \item the functor $\CC$ defines a cooperad ladder;
    \item for every $i \in \alpha$, the restriction of the differential of $\CC^{(i+1)}$ to the planar generators of $\CC_{\mathrm{pl}}^{(i+1)}$ factors through:
      \begin{equation}
        d: \CC_{\mathrm{pl}}^{(i+1)} \otimes 1 \longrightarrow \CC_{\mathrm{pl}}^{(i+1)} \otimes 1 + \CC^{(i)} \to \CC^{(i+1)}.
      \end{equation}
  \end{enumerate}
\end{definition}

\begin{definition}[{\cite[Definition~34]{LeGrignouRocaiLucio2023}}]
  A \emph{quasi-planar cooperad} is a conilpotent dg cooperad which is isomorphic to the colimit of a quasi-planar cooperad ladder.
\end{definition}

An important class of examples is given by bar constructions of operads tensored by the Barratt--Eccles operad.
In essence, this follows from the fact that $\E$ admits a planar basis $\E_{\pl}(n)_r$ spanned by elements of the form $(\id_n, \sigma_1, \dots, \sigma_r)$, where $\sigma_i \in \Sym_n$, and most importantly that this planar basis forms a graded sub-operad of $\E$.
We can thus consider a grading on $\B(\PP \otimes \E)$ by the sum of the degrees minus one of the $\E$ components, which gives rise to a cooperad ladder by considering the associated increasing filtration.

\begin{proposition}[{\cite[Propositions~10, 15]{LeGrignouRocaiLucio2023}}]
  Let $\PP$ be a dg-operad such that $\PP(0) = 0$.
  The bar construction $\B(\PP \otimes \E)$ is a quasi-planar cooperad for the filtration above.
\end{proposition}

See also Proposition~\ref{prop:b-surj-qpl} below for another class of examples of quasi-planar cooperads.

\begin{remark}
  Again, since $\PP(0) = 0$, we have that $\PP \otimes \E = \PP \otimes \Enu$ and both operads are augmented. Let us point out that here, the operadic bar construction $\B$ refers to the one in \cite{LodayVallette12}, which is not exactly the same as the one considered in \cite{LeGrignouRocaiLucio2023}. However, the proof that $\B(\PP \otimes \E)$ is quasi-planar is exactly the same for this bar construction. See also \cite[Proposition 1 and Remark 2]{LeGrignouRocaiLucio2023b}.
\end{remark}

\subsubsection{The comodule structure of cobar constructions of quasi-planar cooperads}

Recall that the dg operad $\E$ is a Hopf operad, meaning it is a coalgebra for the Hadamard tensor product of operads. Its coalgebra structure is given by the following diagonal map:
\begin{equation}\label{eq:hopf-e}
  \Delta_{\E(r)}: \E(r) \to \E(r) \otimes \E(r), \quad
  (\sigma_0, \ldots, \sigma_n) \mapsto \sum_{i=0}^n (\sigma_0, \ldots, \sigma_i) \otimes (\sigma_i, \ldots, \sigma_n),
\end{equation}
where $(\sigma_0, \dots, \sigma_n) \in \E(r)_n = \kk[\Sym_r^{\times(n+1)}]$. Moreover, $\Enu$ is a Hopf suboperad in an obvious way.

\begin{notation}
  Let $M = (M_{\pl} \otimes \Sym, d_M)$ be a quasi-free dg $\Sym$-module.
  For $\sigma \in \Sym_n$, let $D_\sigma$ be the $\Sym_n$-equivariant degree $-1$ endomorphism of $M(n)$ whose restriction to the generators $M_{\pl}$ is the composition:
  \[
    D_\sigma : M_{\pl}(n) \otimes \{\mathrm{id}\} \hookrightarrow M(n)  \xrightarrow{d_M} M(n) \twoheadrightarrow M_{\pl}(n) \otimes \{\sigma \} \hookrightarrow M_{\pl}(n)~.
  \]
  For any $x$ in $M_{\pl}(n)$, we have $D_\sigma(x \otimes \{\mathrm{id}\}) = d_\sigma (x) \otimes \{\sigma \}$.
  Similarly, for a sequence of permutations $\underline \sigma \coloneqq (\sigma_1, \dots, \sigma_k) \in \Sym_n^k$,
  we set:
  \[
    d_{\underline{\sigma}} \coloneqq d_{\sigma_1} \dots d_{\sigma_k} \quad \text{and} \quad D_{\underline{\sigma}} \coloneqq D_{\sigma_1} \dots D_{\sigma_k}~.
  \]
\end{notation}

\begin{proposition}\label{prop:comodule structure}
  Let $\CC$ be a quasi-planar cooperad with $\CC(0) = 0$.
  Then $\Omega \CC$ admits a comodule structure over the dg Hopf operad $\Enu$:
  \[
    \Delta: \Omega \CC \longrightarrow \Enu \otimes  \Omega \CC
  \]
  which, on the generators, is given by the formula:
  \begin{align*}
    \Delta: s^{-1} \CC_\pl & \to \E_\pl \otimes s^{-1}\CC
    \\
    s^{-1}x \otimes \{\id\}                 & \mapsto \sum_{k \geq 0} \sum_{\underline\sigma \in \Sym_n^k}  \rho(\underline\sigma) \otimes D_{\underline\sigma}(s^{-1}x \otimes \{\id\}),
  \end{align*}
  where $\rho(\underline\sigma) = (\id, \sigma_k, \sigma_{k-1}\sigma_k, \dots, \sigma_1 \dots \sigma_k)$, where $s^{-1} x \in s^{-1} \CC_\pl (n)$ is a planar generator of $\CC$ and where the sum is taken over all sequences of permutations $\underline\sigma = (\sigma_1, \dots, \sigma_k) \in \Sym_n^k$ for all $k \geq 0$.
\end{proposition}

\begin{proof}
  The existence of this map is a direct consequence of \cite[Theorem 2]{LeGrignouRocaiLucio2023}.
  Notice that the $\E$-comodule structure of the theorem factors as a $\Enu$-comodule structure since $\CC(0) = 0$.
  We refer to \cite[Section 2.8]{LeGrignouRocaiLucio2023} for more details.
\end{proof}

\begin{remark}
  When $\CC = \B(\sLie \otimes \E)$, then $\CC_\pl = \mathbb{T}^c(\sLie \otimes \E_\pl)$ as a graded $\mathbb{N}$-module.
\end{remark}

\subsubsection{The point-set cotensor over coalgebras}
Since the conilpotent dg cooperad $\B(\sLie \otimes \E)$ is quasi-planar, its cobar construction $\OB(\sLie \otimes \E)$ admits a canonical comodule structure over the dg Hopf operad $\Enu$ by Proposition~\ref{prop:comodule structure}. By the work of Berger and Fresse~\cite{BergerFresse04}, the (reduced) cellular cochain complex $\ch*(X)$ of any simplicial set $X$ admits a canonical $\Enu$-coalgebra structure. The strategy to present the cotensor of spectral Lie algebras by spaces at the point-set level is to consider the tensor product with $\ch*(X)$ and pull it back along the $\Enu$-comodule structure map.

\begin{proposition}\label{prop:tensor bifunctor}
  The tensor product of chain complexes lifts to a functor
  \[
    (- \otimes -): \dg~\Enu\text{-}\alg \times \dg~\OB(\sLie \otimes \E)\text{-}\alg \longrightarrow \dg~\OB(\sLie \otimes \E)\text{-}\alg.
  \]
\end{proposition}

\begin{proof}
  This follows directly from Proposition~\ref{prop:comodule structure}. Indeed, the tensor product of a dg $\Enu$-algebra and a dg $\OB(\sLie \otimes \E)$-algebra is naturally a dg $\Enu \otimes  \OB(\sLie \otimes \E)$-algebra, and pulling back along the comodule structure of the aforementioned proposition endows it with a dg $\OB(\sLie \otimes \E)$-algebra structure.
\end{proof}

\begin{corollary}
  The category of dg $\OB(\sLie \otimes \E)$-algebras is cotensored over pointed simplicial sets using the bifunctor
  \[
    \ch*~(-) \otimes (-) : \mathsf{sSet}^{\op} \times \dg~\OB(\sLie \otimes \E)\text{-}\alg \longrightarrow \dg~\OB(\sLie \otimes \E)\text{-}\alg.
  \]
\end{corollary}

\begin{proof}
  This follows directly by precomposing the bifunctor of Proposition~\ref{prop:tensor bifunctor} with the reduced cellular cochain functor $\ch*~(-)$ and its dg $\Enu$-algebra structure as constructed in \cite{BergerFresse04}.
\end{proof}

\subsubsection{A point-set model for the homotopy cotensor}

The following result is an analogue of Knudsen~\cite[Proposition 3.18]{Knudsen17} (see also Hinich~\cite[Lemma 4.8.3]{Hinich1997}) for the cotensoring of spectral Lie algebras by finite type pointed spaces.
Note that such a cotensorization is essentially unique by \cite[Section~4.4.4]{Lurie2009a}.

\begin{theorem}\label{thm: point-set cotensor with spaces}
  The functor
  \[
    \ch*~(-) \otimes -: \mathsf{sSet}^{\mathsf{f.t.},\op} \times \dg~\OB(\sLie \otimes \E)\text{-}\alg \longrightarrow \dg~\OB(\sLie \otimes \E)\text{-}\alg
  \]
  presents the $\infty$-categorical cotensoring of spectral Lie algebras by finite type pointed spaces, i.e., $\gk^X \simeq \ch*~(X) \otimes \gk$ for any finite type pointed simplicial set $X$ and any dg $\OB(\sLie \otimes \E)$-algebra $\gk$.
\end{theorem}

\begin{remark}
  We could remove the finiteness assumption on $X$ by considering the hom-space $\hom(\ch~(X), \gk)$ instead of $\ch*~(X) \otimes \gk$ and endowing it with a convolution dg $\OB(\sLie \otimes \E)$-algebra structure. However, this would require working with $\Enu$-coalgebras, which is more cumbersome. For simplicity, we will work with algebras instead. The aforementioned convolution structure can be constructed using the results in \cite{LeGrignou2020a,LeGrignou2022}.
\end{remark}

\begin{proof}
  Let us use the general framework of simplicial frames from~\cite[II(a), Section 3.2.7]{Fresse2017a}; our proof is similar to the one of~\cite[Theorem~7.1.5]{Fresse2017a}.
  Let us first deal with unpointed spaces.
  Consider the Reedy model structure on the category of simplicial $\OB(\sLie \otimes \E)$-algebras (see e.g., \cite[Section 3.1]{Fresse2017a} and the references therein).
  To apply the framework of simplicial frames, we need to check the following two criteria:
  \begin{itemize}
    \item The projections $\epsilon_n : \Delta^n \to \Delta^0$ induce weak equivalences for all $n \geq 0$:
      \[
        \epsilon^* : \gk = \ch*(\Delta^0) \otimes \gk \longrightarrow \ch*(\Delta^n) \otimes \gk.
      \]
    \item The inclusions $\iota_n : \operatorname{sk}_0 \Delta^n \hookrightarrow \Delta^n$ (where $\operatorname{sk}_0 \Delta^n = \bigsqcup_{k=0}^n \Delta^0$ is the $0$-skeleton of $\Delta^n$) induce Reedy fibrations for all $n \geq 0$:
      \[
        \iota_n^* : \ch*(\Delta^n) \otimes \gk \longrightarrow \prod_{k=0}^n \ch*(\Delta^0) \otimes \gk = \gk^{\oplus (n+1)}.
      \]
  \end{itemize}

  The first point is clear since the tensor product over a field preserves quasi-isomorphisms and $\ch*(\Delta^n)$ is acyclic.
  For the second point, note that fibrations of $\OB(\sLie \otimes \E)$-algebras are fibrations of chain complexes, i.e., degreewise surjections.
  Moreover, the tensor product is right exact, so we can reduce to checking the fibration property on cellular cochains.

  Since $\ch*$ sends colimits to limits, the matching space is $\ch*(\partial \Delta^n)$.
  It follows that all we need to check is that the map $\ch*(\Delta^n) \longrightarrow \ch*(\partial \Delta^n)$ is a degreewise surjection (a property often called ``extendability''~\cite[§10]{FelixHalperinThomas2001}).
  This is clear, since that map is simply the projection that kills the top cochain.
  Hence, the second point is also satisfied.

  Now, we can apply the framework of simplicial frames to deduce that the functor $\Delta^n \mapsto \ch*(\Delta^n) \otimes \gk$ is a Reedy fibrant resolution of $\gk$ in the category of $\OB(\sLie \otimes \E)$-algebras.
  Taking fibers, we deduce the statement of the theorem for pointed spaces.
\end{proof}

\subsection{Point-set version of Knudsen's theorem}

We put together the results of the previous sections to exhibit specific chain complex lifts of Knudsen's theorem (smashed with $H\kk$ to obtain homology).

\begin{theorem}\label{thm:main}
  Let $M$ be a parallelizable $d$-manifold of finite type, and let $V$ be a chain complex over $\kk$. There is a zig-zag of quasi-isomorphisms of weighted chain complexes over $\kk$:
  \[
    \bigoplus_{k \geq 1} \ch(\Conf_k(M)) \otimes_{h\Sym_k} V^{\otimes k} \simeq \B_\iota \bigl( \ch*~(M^+) \otimes \OB(\sLie \otimes \E) \circ (s^d V)\bigr),
  \]
  where, on the left-hand side, the weight is given by $k$, and on the right-hand side, the weight is given by considering $V$ to be in weight $1$.
\end{theorem}

\begin{proof}
  This follows directly from Proposition~\ref{prop:rectification of spectral Lie algebras}, Proposition~\ref{prop: computation of derived indecomposable} and Theorem~\ref{thm: point-set cotensor with spaces}, which lift the constructions that appear in Theorem~\ref{thm:knudsen} to the point-set level.
\end{proof}

\begin{remark}
  The finite type assumption in Theorem~\ref{thm:main} can be removed by working with chains instead of cochains and considering their $\Enu$-coalgebra structure. In practice, however, the finite type assumption is almost always satisfied and it is easier to work with cochains and their $\Enu$-algebra structure instead.
\end{remark}

\section{Smaller models and computational aspects}

In this section, we discuss the computational aspects of the chain complexes exhibited in the previous section. First, we show that they can be reduced by replacing the Barratt--Eccles dg operad with the surjections operad $\Surj$, which is another model for the $\EE_\infty$ operad. Along the way, we show that the bar construction $\B(\PP \otimes \Surj)$ is \emph{quasi-planar} in the sense of \cite{LeGrignouRocaiLucio2023}, for any dg operad $\PP$ (Proposition~\ref{prop:b-surj-qpl}).
We also explain how this complex can be filtered by two distinct filtrations: the filtration induced by the quasi-planar filtration of $\B(\PP \otimes \Surj)$ and the filtration induced by the complexity filtration on $\Surj$ (which corresponds to the filtration given by the different $\EE_n$ suboperads).

\medskip

Finally, we work out some examples of these chain complexes by explicitly computing the algebra structure over the surjections operad of $\ch*(M^{+})$ when $M = \mathbb{R}^d$ and recover many classical results. In general, the key ingredient for carrying these computations further is the computation of the surjection-algebra structure of $\ch*(M^{+})$.

\subsection{A simplification of the complex using the surjection operad}

The Barratt--Eccles operad $\E$ is quasi-isomorphic to the surjection operad $\Surj$~\cite{McClureSmith2002,McClureSmith2003,BergerFresse04}, which is a quotient of $\E$.
Using this smaller model, we can simplify the right-hand side of Theorem~\ref{thm:main}.

\begin{definition}
  The surjection operad $\Surj$ is the dg operad defined as follows.
  As a graded $\Sym$-module, $\Surj(r)_d$ is spanned by maps $u : \underline{r+d} \to \underline{r}$; we mod out by the submodule spanned by degenerate maps, i.e., those that are not surjective or that admit two consecutive elements with the same image.
  We denote by $(u_1, \dots, u_{r+d})$ the surjection sending $i$ to $u_i$, or even $u_1 \dots u_{r+d}$ when there is no ambiguity.

  The symmetric group $\Sym_r$ acts on $\Surj(r)_d$ by post-composition.
  The differential is given by summing over all ways of removing an element in the domain of a surjection, i.e.,
  \begin{equation}
    d(u_1, \dots, u_{r+d}) = \sum_{i=1}^{r+d} \pm (u_1, \dots, \widehat{u_i}, \dots, u_{r+d}),
  \end{equation}
  with signs\footnote{\Textcite{BergerFresse04} and \textcite{McClureSmith2002,McClureSmith2003} have differing sign conventions for the surjection operad. We use the convention from \cite{BergerFresse04} throughout this paper.} determined in~\cite[Section~1.2.3]{BergerFresse04}.
  The operadic composition is given by inserting a surjection into another one and reindexing the elements of the domain accordingly; we refer to~\cite[Section~1.2.4]{BergerFresse04} for more details including signs.
\end{definition}

\begin{theorem}[{\cite[Theorem~1.3.2]{BergerFresse04}, \cite{McClureSmith2002,McClureSmith2003}}]
  There is a quasi-isomorphism of dg operads, called the table reduction morphism:
  \[
    \TR: \E \qi \Surj.
  \]
\end{theorem}

Obviously, the table reduction map restricts to a quasi-isomorphism $\TR: \Enu \qi \Surjnu$ between their non-unital versions. Let $\iota' : \B(\sLie \otimes \Surj) \longrightarrow \OB(\sLie \otimes \Surj)$ be the universal twisting morphism.
In order to make sense of the next theorem, we first need to explain how $\ch*~(M^+) \otimes \OB(\sLie \otimes \Surj) \circ (s^d V)$ is an algebra over $\Surjnu \otimes \OB(\sLie \otimes \Surj)$. For this, we prove (Proposition~\ref{prop:b-surj-qpl}) that $\B(\sLie \otimes \Surj)$ is a quasi-planar cooperad, and thus that its cobar construction $\OB(\sLie \otimes \Surj)$ admits a comodule structure over the Hopf operad $\Enu$ by Proposition~\ref{prop:comodule structure}.
Finally, we will use the results of \cite{BergerFresse04,McClureSmith2002,McClureSmith2003} which state that the $\Enu$-algebra structure on $\ch*~(M^+)$ factors through $\Surjnu$.

\begin{theorem}\label{thm:main-surj}
  There is a quasi-isomorphism of weighted chain complexes over $\kk$:
  \[
    \B_\iota \bigl( \ch*~(M^+) \otimes \OB(\sLie \otimes \E) \circ (s^d V)\bigr) \simeq \B_{\iota'} \bigl( \ch*~(M^+) \otimes \OB(\sLie \otimes \Surj) \circ (s^d V)\bigr),
  \]
  where, on both sides, the weight is given by considering $V$ to be in weight $1$.
\end{theorem}

The following proposition mirrors~\cite[Proposition~10]{LeGrignouRocaiLucio2023} and is proved in a similar manner.

\begin{proposition}\label{prop:b-surj-qpl}
  Let $\PP$ be a dg operad with $\PP(0) = 0$.
  The conilpotent dg cooperad $\B(\PP \otimes \Surj)$ is quasi-planar.
\end{proposition}

\begin{remark}
  The case of a general dg operad $\PP$ involves curved cooperads and curved bar constructions.
  The quasi-planar structure in this more general setting is similar to the one presented here, but requires careful handling of the curvature terms.
\end{remark}

To define the quasi-planar structure, let us adapt the methods of~\cite[Section~2.9]{LeGrignouRocaiLucio2023}.

\begin{definition}
  Let $r,d \geq 0$ be integers.
  A non-degenerate map $u : \underline{r+d} \to \underline{r}$ in $\Surj(r)_d$ is said to be \emph{planar} if the first occurrences of $1, \dots, r$ in the sequence $(u_1, \dots, u_{r+d})$ appear in increasing order.
  We denote by $\Surj_{\pl}(r)_d$ the submodule of $\Surj(r)_d$ spanned by planar surjections.
\end{definition}

\begin{example}
  The element $12132 \in \Surj(3)_2$ is planar, whereas $21312 \in \Surj(3)_2$ is not planar.
\end{example}

Clearly, we have $\Surj = \Surj_{\pl} \otimes \Sym$, i.e., planar surjections form a set of planar generators of the surjection operad.
Unlike what happens for the Barratt--Eccles operad, the submodule $\Surj_{\pl}$ is \emph{not} a sub-operad of $\Surj$.
However, we have that:
\begin{lemma}\label{lem:surj-planar-compos}
  Let $u \in \Surj_{\pl}(r)$ and $v \in \Surj_{\pl}(s)$.
  Then $u \circ_r v$ is planar.
\end{lemma}

\begin{example}
  Consider the two elements $121 \in \Surj_{\pl}(2)_1$ and $12 \in \Surj_{\pl}(2)_0$.
  Then $121 \circ_1 12 = 1312 + 1232$, and $1312 \in \Surj(3)_1$ is not planar.
  On the other hand, $121 \circ_2 12 = 1231$ is planar.
\end{example}

\begin{remark}
  Recall that $\E_{\pl}(r)_d$ is spanned by non-degenerate sequences $\underline{\sigma} = (\sigma_0, \dots, \sigma_d)$ such that $\sigma_0 = \id_r$.
  In general, it is \emph{not} true that $\TR(\E_{\pl}(r)_d)$ is equal to $\Surj_{\pl}(r)_d$.
  In fact, it is not possible to choose a planar basis of $\Surj$ compatible with the table reduction morphism.
  E.g., the rank of $\Surj(3)_1$ is $18$, so a planar basis should have cardinality $18/3! = 3$; but $\TR(\E_{\pl}(3)_1)$ has rank $5$ (spanned by $1232$, $1213$, $1231$, $1312$, $1321$).

  On the other hand, we do have $\Surj_{\pl}(r)_d \subseteq \TR(\E_{\pl}(r)_d)$.
  Indeed, given $u \in \Surj(r)_d$, the algorithm described in~\cite[Section~1.4.2]{BergerFresse04} to construct a preimage of $u$ under $\TR$ always produces an element of $\E_{\pl}(r)_d$ when $u$ is planar.
\end{remark}

Lemma~\ref{lem:surj-planar-compos} motivates the following definition.
Consider the degree filtration on $\Surj_{\pl}$:
\begin{equation}
  F_n \Surj_{\pl}(r)_d =
  \begin{cases}
    \Surj_{\pl}(r)_d, & \text{if } d \leq n-1; \\
    0,                & \text{otherwise}.
  \end{cases}
\end{equation}
This induces a filtration on $\Surj_{\pl} \otimes \PP$ for any operad $\PP$.

Since $\Surj$ is arity-wise free as a module over the symmetric group, we obtain $\B(\Surj \otimes \PP) = \B_{\pl}(\Surj_\pl \otimes \PP) \otimes \Sym$ as a graded symmetric sequence, where $\B_{\pl}(\Surj_\pl \otimes \PP)$ is spanned by planar trees decorated by elements of $\Surj_{\pl} \otimes \PP$.
We call an edge of a planar tree \emph{bad} if it is internal (i.e., joins two vertices) and it is not the last edge among the children of its parent.
We thus filter $\B_{\pl}(\Surj_{\pl} \otimes \PP)$ by the number of bad edges plus the degree filtration on $\Surj$, which induces a filtration on $\B(\Surj \otimes \PP)$.

\begin{example}
  Let $u \in \Surj(2)_{i-1}$ and $v \in \Surj(r)_{j-1}$.
  Then the element $u(v, \id) \in (\B_{\pl}\Surj)(r+1)_{i+j}$ is in filtration level $i+j+1$, whereas $u(\id, v)$ is in filtration level $i+j$.
\end{example}

\begin{proof}[Proof of Proposition~\ref{prop:b-surj-qpl}]
  Let us check the hypotheses of~\cite[Definition~35]{LeGrignouRocaiLucio2023}.
  The filtration $\{F_n \B(\Surj \otimes \PP)\}$ exhausts $\B(\Surj \otimes \PP)$.
  Moreover, the reduced cooperad coproduct $\bar\Delta : F_n\B(\Surj \otimes \PP) \to \mathcal{T}^c(\B(\Surj \otimes \PP))$ factors through $\mathcal{T}^c(F_{n-1}\B(\Surj \otimes \PP))$ as $s\Surj$ is concentrated in positive degrees and the bar comultiplication does not create edges, and thus does not create bad edges.
  Finally, the non-planar component of the differential either comes from the internal differential of $\Surj$, which lowers the internal degree and thus the filtration; or from the contraction of a bad edge (Lemma~\ref{lem:surj-planar-compos}), which lowers the filtration as well.
\end{proof}

\begin{proof}[Proof of Theorem~\ref{thm:main-surj}]
  We can now mimic the proof of Theorem~\ref{thm:main}, replacing $\E$ by $\Surj$ and using Proposition~\ref{prop:b-surj-qpl} instead of Proposition~\ref{prop:rectification of spectral Lie algebras}.
\end{proof}

\begin{remark}
  There is a linear map from the LHS to the RHS of Theorem~\ref{thm:main-surj} induced by the quasi-isomorphism of dg operads $\TR : \E \to \Surj$.
  However, this map is \emph{not} a chain map in general.
  The main issue is that the $\E$-comodule structures of $\B(\PP \otimes \E)$ and $\B(\PP \otimes \Surj)$ are not compatible, that is, the following diagram does \emph{not} commute:
  \[
    \begin{tikzcd}
      \B(\PP \otimes \E) \arrow{rr}{\Delta_\E} \arrow{d}{\TR} \arrow[drr, phantom, "\text{\scriptsize {\fontencoding{U}\fontfamily{futs}\selectfont\char 49\relax} not commutative!}"] && \E \otimes \B(\PP \otimes \E) \arrow{d}{\TR \otimes \TR} \\
      \B(\PP \otimes \Surj) \arrow{r}{\Delta_\Surj}  & \E \otimes \B(\PP \otimes \Surj) \arrow{r}{\TR}             & \Surj \otimes \B(\PP \otimes \Surj).
    \end{tikzcd}
  \]
  For a concrete example, consider $\PP = \Com$ so that $\PP \otimes \E = \E$ and $\PP \otimes \Surj = \Surj$.
  Consider the planar element $[123|132|321] \in \E(3)_2$ viewed as a generator of $\B\E$.
  Its table reduction is zero, but applying the comodule structure first and then the table reduction gives $1232 \otimes s(1321) \neq 0$.
  Nonetheless, our arguments in the proof of Theorem~\ref{thm:main-surj} show that the two weighted chain complexes are quasi-isomorphic (albeit not directly).
\end{remark}

\subsection{An explicit chain complex and its filtrations}
Let $M$ be a parallelizable $d$-manifold of finite type, and let $V$ be a chain complex over $\kk$. As a direct consequence of Theorems~\ref{thm:main} and~\ref{thm:main-surj}, the homology of $\ch(\Conf_k(M)) \otimes_{h \Sym_k} V^{\otimes k}$ is computed by the following complex
\[
  \B_\iota^{(k)} \bigl( \ch*~(M^+) \otimes \OB(\sLie \otimes \Surj) \circ (s^d V)\bigr) =
\]
\begin{equation}\label{eq: the chain complex for configurations of k points}
  \bigoplus_{n \geq 1} \B(\sLie \otimes \Surj)(n) \otimes_{\Sym_n} \left[ \ch*~(M^+) \otimes \left( \OB(\sLie \otimes \Surj)(k) \otimes_{\Sym_k} (s^d V)^{\otimes k} \right) \right]^{\otimes n}
\end{equation}
for any $k \geq 1$. Let us analyze the different terms in the differential of this complex.

\begin{enumerate}
  \item For every $n \geq 1$, the complex $\B(\sLie \otimes \Surj)(n)$ has an internal differential, corresponding to the bar construction of the dg-operad $\sLie \otimes \Surj$. This complex on its own is a projective resolution of the trivial representation of $\Sym_n$ since $\B(\sLie \otimes \Surj) \qi \Com$.
  \item The complex $\OB(\sLie \otimes \Surj)(k) \otimes_{\Sym_k} (s^d V)^{\otimes k}$ has an internal differential; its homology is given by the arity $k$ component of the homology of the free spectral Lie algebra generated by $V$, see~\cite{antolin-camarena_mod_2020,kjaer_odd_2018}.
  \item The complex $\ch*~(M^+)$ has an internal differential, which computes the homology of $M^+$ with $\kk$ coefficients.
  \item The most important term of the differential is the term that relates these two previous blocks: it takes a leveled tree in $\B(\sLie \otimes \Surj)(n)$, cuts the top level, deshifts it, and then makes it act on
    \[
      \left[ \ch*~(M^+) \otimes \left( \OB(\sLie \otimes \Surj)(k) \otimes_{\Sym_k} (s^d V)^{\otimes k} \right) \right]^{\otimes n};
    \]
    this action is fully determined by the $\Surjnu$-algebra structure of $\ch*~(M^+)$ together with the comodule map of Proposition~\ref{prop:comodule structure}.
\end{enumerate}

Through different filtrations on the total complex, our goal is to analyze this fourth term of the differential that determines the homology of $\ch(\Conf_k(M)) \otimes_{h \Sym_k} V^{\otimes k}$.

\begin{remark}
  In particular, it is apparent from the above description that the homology of these configuration spaces only depends on the $\EE_\infty$-algebra structure of $\ch*~(M^+)$. This is a particular case of~\cite[Corollary 5.16]{Petersen20} when the base ring is a field and the manifold is parallelizable.
\end{remark}

Let us illustrate the previous description of the complex and the differentials.
A typical element of $\B_\iota^{(k)} \bigl( \ch*~(M^+) \otimes \OB(\sLie \otimes \Surj) \circ (s^d V)\bigr)$ is a linear combination of trees as in Figure~\ref{fig:example-tree-complex}.
Each part of the differential described above corresponds to a different way of modifying the tree:
\begin{enumerate}
  \item The internal differential of $\B(\sLie \otimes \Surj)(n)$ corresponds to modifying the red tree at the bottom of the picture, either by applying the differential of $\sLie \otimes \Surj$ to the labels of the vertices or by contracting an edge.
  \item The internal differential of $\OB(\sLie \otimes \Surj)(k) \otimes_{\Sym_k} (s^d V)^{\otimes k}$ corresponds to modifying the black subtrees, either by applying the differential of $\sLie \otimes \Surj$ to the labels of the vertices, by contracting a solid edge (bar construction), or by converting a solid edge into a dashed edge (cobar construction).
  \item The internal differential of $\ch*~(M^+)$ corresponds to modifying the blue boxes, by applying the differential of $\ch*~(M^+)$ to the labels of the vertices.
  \item The term of the differential that relates the red tree to the black trees corresponds to cutting the red tree at the top level, deshifting it, and making it act on the tensor product of the blue box and the black trees according to the $\Surjnu$-algebra structure of $\ch*~(M^+)$ and the comodule structure of $\B(\sLie \otimes \Surj)$.
\end{enumerate}

\begin{figure}[htbp]
  \centering
  \begin{forest} uconf tree
    [{$\left[x_{1}, x_{2}\right]$}, rv[{1}, lf, re][{2}, lf, re]]
  \end{forest}
  \caption{An example of an element in the complex $\B_\iota^{(k)} \bigl( \ch*~(M^+) \otimes \OB(\sLie \otimes \Surj) \circ (s^d V)\bigr)$. The bottom red tree is an element of $\B(\sLie \otimes \Surj)$, the blue boxes are elements of $\ch*~(M^+)$, and the black trees are elements of $\OB(\sLie \otimes \Surj)(k) \otimes_{\Sym_k} (s^d V)^{\otimes k}$. Black solid edges are within $\B(\sLie \otimes \Surj)$, whereas dashed edges correspond to the composition in the cobar construction.}
  \label{fig:example-tree-complex}
\end{figure}

\subsubsection{The complexity filtration}\label{subsection: the complexity filtration}

We now recall the complexity filtration of the surjection operad $\Surj$ introduced in~\cite[Section~3]{McClureSmith2003}.

\begin{definition}
  Let $u$ be a finite sequence of integers.
  If $u$ only takes values $i$ and $j$ for some $1 \leq i < j \leq r$, we define its \emph{complexity} as the number of times $u$ switches from $i$ to $j$ or from $j$ to $i$.
  In general, the \emph{complexity} of $u$ is defined as the maximum of the complexities of the sequences of the form $u|_{i,j}$, where $u|_{i,j}$ is obtained from $u$ by removing all values except $i$ and $j$.

  Let $k \geq 1$ be an integer. We let $\Surj_k$ be the sub-$\Sym$-module of $\Surj$ spanned by surjections of complexity at most $k$.
\end{definition}

\begin{theorem}[{\cite[Theorem~3.5]{McClureSmith2003}}]\label{thm:complexity-filt}
  The collection $\{\Surj_k\}_{k \geq 1}$ forms a filtration of dg-operads of $\Surj$:
  \[
    \Surj_1 \subseteq \Surj_2 \subseteq \dots \subseteq \Surj_k \subseteq \Surj_{k+1} \subseteq \dots \subseteq \Surj,
  \]
  and there are quasi-isomorphisms of dg-operads $\Surj_k \qi \EE_k$ for all $k \geq 1$.
\end{theorem}

This filtration restricts to $\Surjnu$, where $\Surjnu_k$ is now a model for the $\EEnu_k$ operad. This filtration induces a filtration on the Hadamard tensor product $\sLie \otimes \Surj$ simply by setting:
\[
  F_k(\sLie \otimes \Surj) \coloneqq \sLie \otimes \Surj_k
\]
for all $k \geq 1$. This filtration, in turn, induces a complete increasing filtration on the bar construction $\B(\sLie \otimes \Surj)$ by setting:
\[
  F_n\B(\sLie \otimes \Surj)(m) \coloneqq \bigoplus_t \sum_{i_1+ \dots + i_p = n} \bigotimes_{j=1}^p sF_{i_j}(\sLie \otimes \Surj)(l_j),
\]
for all $n \geq 1$, where the first sum is taken over the planar trees $t$ with $m$ leaves and $p$ nodes whose arities are $l_1, \dots, l_p \geq 2$. The identity and nodes in $t$ labelled by the identity are considered in filtration degree $0$. Finally, this induces a complete filtration on the chain complex~\eqref{eq: the chain complex for configurations of k points} which is compatible with all the differentials and gives rise to an associated spectral sequence.

\begin{remark}
  There are, in fact, multiple ways of filtering the complex~\eqref{eq: the chain complex for configurations of k points} using the complexity filtration on $\Surj$.
  For instance, we could have defined the filtration on $\B(\sLie \otimes \Surj)$ by setting $F_n\B(\sLie \otimes \Surj)(m)$ to be the direct sum of all the terms in the bar construction where at least one of the $\Surj$-labels is in filtration degree at least $n$.
  This would give a different filtration, but it would still be compatible with the differentials and would give rise to a spectral sequence converging to the same homology.
\end{remark}

\begin{remark}
  This filtration and its associated spectral sequence could prove useful, for instance, in the case when $M = \mathbb{R}^n$, since $M^+ = S^n$ and $\ch*(S^n)$ is trivial as an $\EE_n$-algebra but not as an $\EE_{n+1}$-algebra. See Section~\ref{subsec: homology of unordered configurations of Rn} for more details on how this relates to known computations of the homology of the (unordered) configuration spaces of $\mathbb{R}^n$.
\end{remark}

\subsection{The homology of unordered configuration spaces of Euclidean spaces}\label{subsec: homology of unordered configurations of Rn}

Let $S^d = \Delta^d / \partial \Delta^d$ be the simplicial model for the $d$-sphere with only two non-degenerate simplices. Let us compute the $\Surj$-algebra structure of $\ch*~(S^d) \cong \kk\langle a_d \rangle$ for all $d \geq 2$.

\begin{definition}\label{def:sphere-admissible}
  A surjective map $u : \underline{n+k} \to \underline{n}$ is \emph{$S^d$-admissible} if:
  \begin{itemize}
    \item the degree $k$ is equal to $d(n-1)$;
    \item there exist $(d+1)$ ``overlapping'' permutations $\sigma_j \in \Sym_n$ for $0 \leq j \leq d$, where $\sigma_j(n) = \sigma_{j+1}(1)$, such that
      \[
        u = (\sigma_0(1), \dots, \sigma_0(n-1), \sigma_1(1), \dots, \sigma_{d-1}(n-1), \sigma_d(1), \dots, \sigma_d(n)),
      \]
      that is, $u$ is obtained by concatenating the first $(n-1)$ values of the first $d$ permutations and the full sequence of the last permutation.
  \end{itemize}
  If $u$ is $S^d$-admissible, then we let $\sigma_i[u]$ be the permutation $\sigma_i$ appearing in the definition of $S^d$-admissibility.
\end{definition}

\begin{proposition}[Cf.~{\cite[Proposition~3.2.5]{BergerFresse04}}]\label{prop:surj-alg-sphere}
  The $\Surjnu$-algebra structure of $\ch*~(S^d) \cong \kk\langle a_d \rangle$ is determined, for any surjection $u$ of arity $n$ and degree $k$ in $\Surjnu(n)_k$, by the following operations on $\ch*~(S^d)$.
  Let $u \in \Surjnu(n)_k$ be a basis element, and let $\mu_u : \ch*~(S^d)^{\otimes n} \to \ch*~(S^d)$ be the induced operation.
  If $u$ is $S^d$-admissible, then the product is given by
  \[
    \mu_u([a_d], \dots, [a_d]) = \epsilon [a_d],
  \]
  where $\epsilon \in \{-1,1\}$ is given by:
  \[
    \epsilon = (-1)^{d \binom{n}{2} + \binom{n}{2} \binom{d}{2}} \prod_{j=0}^{d} \operatorname{sign}(\sigma_j[u]).
  \]
  Otherwise, if $u$ is not $S^d$-admissible, then $\mu_u$ is the zero map on $\ch*~(S^d)$.
\end{proposition}

\begin{proof}
  The fact that only surjections that are concatenations of $(d+1)$ permutations act non-trivially follows directly from the description of the $\Surj$-coalgebra structure on $\ch~(S^d)$ given in~\cite[Section 2]{BergerFresse04}. The sign appearing in this coalgebra structure can be computed as follows: since the table reduction morphism $\TR: \E \to \Surj$ involves no signs, it has to agree with the sign of the $\E$-coalgebra structure on $\ch~(S^d)$, given by pulling back the $\Surj$-coalgebra structure along the table reduction map. Thus this sign is given by
  \[
    \epsilon' = (-1)^{\frac{d(d-1)}{2} \cdot \frac{n(n-1)}{2}} \prod_{j=0}^{d} \operatorname{sign}(\sigma_j[u]),
  \]
  using~\cite[Lemma 2.19]{RocaiLucio2023}.
  This gives the second part of the sign exponent.
  Finally, since we want to compute the $\Surj$-algebra structure on $\ch*~(S^d)$, we take the linear dual of the aforementioned $\Surj$-coalgebra structure on $\ch~(S^d)$ and the sign
  \[
    \lambda = (-1)^{d\frac{n(n-1)}{2}}
  \]
  appears when applying the duality pairing. Hence $\epsilon = \epsilon' \lambda$.
\end{proof}

\begin{example}
  Let us consider the element $u = (12121)$ in $\Surj(2)_3$.
  It is $S^3$-admissible, with $\sigma_0[u] = \sigma_2[u] = \id_2$, $\sigma_1[u] = \sigma_3[u] = (12)$.
  It acts on $\ch(S^3)$ as follows:
  \[
    \mu_{(12121)}([a_3],[a_3]) = [a_3];
  \]
  this is the first non-trivial product in the $\Surj$-algebra structure of $\ch(S^3)$.
\end{example}

From this explicit description of the $\Surj$-algebra structure of $\ch*(S^d)$, we can deduce the following remark about the complexity filtration on $\Surj$ and its action on the cochains of the sphere.
This is a special case of~\cite[Theorem~4.1]{HeutsLand2024}.

\begin{corollary}
  Let $d \geq 0$ be an integer.
  The $\Surjnu_d$-algebra structure of $\ch*(S^d)$ is trivial.
\end{corollary}

\begin{proof}
  Suppose that $u \in \Surj(n)$ acts nontrivially on $\ch*(S^d)$.
  By Proposition~\ref{prop:surj-alg-sphere}, the degree of $u$ is $d(n-1)$ and there exist $(d+1)$ overlapping permutations $\sigma_0, \dots, \sigma_d \in \Sym_n$, with $\sigma_j(n) = \sigma_{j+1}(1)$, expressing $u$ as the $S^d$-admissible surjection of Definition~\ref{def:sphere-admissible}.
  Let us show that the complexity of $u$ is at least $d+1$, from which the result follows.
  In fact, let us show the slightly stronger result (by induction on $d$): for all pairs $(i,j)$ with $1 \leq i < j \leq n$, the complexity of $u|_{i,j}$, the restriction of $u$ to the preimage of $\{i,j\}$, is at least $d+1$.

  If $d = 0$, then $u$ is a permutation of $\{1, \dots, n\}$ and $u|_{i,j}$ is either $ij$ or $ji$, which have both complexity $1 > 0$.
  Now, suppose that the result holds for $d-1$.
  The surjection $u$ is obtained from the $S^{d-1}$-admissible surjection $u'$ built from $\sigma_0, \dots, \sigma_{d-1}$ (Definition~\ref{def:sphere-admissible}) by appending $\sigma_d(2), \dots, \sigma_d(n)$, since $\sigma_d(1) = \sigma_{d-1}(n)$ is already the last entry of $u'$.
  By the induction hypothesis, the complexity of $u'|_{i,j}$ is at least $d$.
  Therefore, all we need to check is that $u|_{i,j}$ has at least one more alternation between $i$ and $j$ than $u'|_{i,j}$.
  There are a few cases to consider, depending on where the values $i$ and $j$ appear in the permutations $\sigma_{d-1}$ and $\sigma_d$:
  \begin{itemize}
    \item If $\sigma_{d-1} = (\dots, i, \dots, j, \dots, k)$ (i.e., $i$ appears before $j$ and $j$ is not the last entry), then $\sigma_d$ must be of the form $(k, \dots, j, \dots, i, \dots)$ or $(k, \dots, i, \dots, j, \dots)$.
      Therefore $u|_{i,j}$ is either $\dots ijji$ or $\dots ijij$, which has either one or two more alternations than $u'|_{i,j}$.
      The case where $\sigma_{d-1}$ has $j$ before $i$ and $i$ is not the last entry is symmetric.
    \item If $\sigma_{d-1} = (\dots,i, \dots, j)$ has $j$ as the last entry, then $\sigma_d = (j, \dots, i, \dots)$ must have $j$ as the first entry and $i$ afterwards.
      Thus, $u|_{i,j}$ is $\dots j i j$ which has one more alternation than $u'|_{i,j}$.
      The case where $\sigma_{d-1}$ has $i$ as the last entry is symmetric.
  \end{itemize}
  In any case, the complexity of $u|_{i,j}$ is at least $d+1$.
\end{proof}

The software computation is detailed in Appendix~\ref{sec:computational-aspects}.

\subsection{The homology of unordered configuration spaces of the torus}
\label{sec:homology-of-unordered-configurations-of-the-torus}

In this subsection, we compute the Barratt--Eccles algebra structure of $\ch*(S^1) \otimes \ch*(S^1)$, which is a model for the $\EE_\infty$-algebra of cochains on the torus $T = S^1 \times S^1$. From this computation, we obtain a fully explicit chain complex whose homology is the homology of unordered configuration spaces of the torus, which we then use to derive concrete computational results.

\subsubsection{The $\EE_\infty$-algebra structure of the unreduced circle}
We have that, as a chain complex, $\ch*(S^1) \cong \kk.[*] \oplus \kk.[a_1]$, where $[*]$ is the class corresponding to the base point. To compute its algebra structure over the Barratt--Eccles operad, we first introduce the following combinatorial function.

\medskip

Let $\underline{\sigma} = (\sigma_0, \cdots, \sigma_{s-1})$ be an $s$-tuple of permutations in $\mathbb{S}_n$, for some $n \geq 1$ and $s \geq 0$. This $s$-tuple corresponds to an arity $n$ element of degree $s-1$ in the Barratt--Eccles operad. Let $I_S \subset \{1, \cdots, n\}$ be a subset of size $s$; let us denote its elements by $I_S = \{\alpha_1,\cdots, \alpha_s\}$ with the induced order from the inclusion. We denote by $\sigma_j(I_S)$ the first occurrence of an element of $I_S$ in the permutation $\sigma_j$. This gives an $s$-tuple of elements of $I_S$:
\[
  \underline{\sigma}(I_S) = (\sigma_0(I_S),\cdots,\sigma_{s-1}(I_S))~.
\]
We define the function $\psi_{I_S}: \mathbb{S}_n^{\times s} \longrightarrow \{-1,0,1\}$ as follows: if $\underline{\sigma}(I_S)$ is an $s$-tuple of \emph{distinct} elements of $I_S$, that is, if it defines a permutation of the set $I_S$ (up to identifying it with $\{1,\cdots,s\}$ via the unique order-preserving bijection), then $\psi_{I_S}(\underline{\sigma}) = \mathrm{sign}(\underline{\sigma}(I_S))$; and if $\underline{\sigma}(I_S)$ contains the same element twice, then $\psi_{I_S}(\underline{\sigma}) = 0$.

\begin{proposition}\label{prop: unreduced circle}
  The $\E$-algebra structure of $\ch*(S^1) \cong \kk.[*] \oplus \kk.[a_1]$ is determined, for any element $\underline{\sigma} = (\sigma_0, \cdots, \sigma_{s-1})$ of arity $n$ and degree $s-1$ in $\E(n)_{s-1}$, by the following operations:
  \[
    \mu_{\underline{\sigma}}([*], \cdots, [a_1], \cdots, [*], \cdots, [a_1]) = (-1)^{\frac{s(s-1)}{2}} \psi_{I_S}(\underline{\sigma}) [a_1]~,
  \]
  where there are exactly $s$ copies of $[a_1]$ in the inputs and where $I_S$ is the subset of $\{1,\cdots,n\}$ determined by the positions of the inputs equal to $[a_1]$.
\end{proposition}

\begin{proof}
  This follows from the computation of the $\E$-coalgebra structure on $\ch(\Delta^1)$ in \cite[Proposition 2.14]{RocaiLucio2023}, first by postcomposing the decompositions of $[a_1]$ along the projection $\Delta^1 \twoheadrightarrow S^1$ that sends $[0]$ and $[1]$ to $[*]$, and then by applying linear duality to obtain the explicit $\E$-algebra structure on $\ch*(S^1)$.
\end{proof}

\begin{remark}
  We compute the $\E$-algebra structure because $\E$ is a Hopf operad (unlike the surjection operad $\Surj$), which in particular means that $\ch*(S^1) \otimes \ch*(S^1)$ has a canonical $\E$-algebra structure via the diagonal map of $\E$.
\end{remark}

\begin{example}
  When $I_S = \{1,\cdots,n\}$, $\psi_{I_S}(\underline{\sigma})$ (when nonzero) is given by the sign of the permutation $(\sigma_0(1), \cdots,\sigma_{n-1}(1))$, which recovers the formula of \cite[Theorem 3.2.4]{BergerFresse04}.
\end{example}

\subsubsection{The $\EE_\infty$-algebra structure of $\ch*(S^1) \otimes \ch*(S^1)$}
Keeping the same notation as in the previous subsection, let us write down an explicit basis of $\ch*(S^1) \otimes \ch*(S^1)$ as a chain complex as follows:
\[
  [0] = [*] \otimes [*] \quad [\alpha] = [a_1] \otimes [*], \quad [\beta] = [*] \otimes [a_1] \quad \text{and} \quad [\gamma] = [a_1] \otimes [a_1].
\]
\begin{proposition}\label{prop: torus algebra}
  The $\E$-algebra structure of $\ch*(S^1) \otimes \ch*(S^1)$ is determined as follows. Let $\underline{\sigma} = (\sigma_0, \cdots, \sigma_{k-1})$ be an element of arity $n$ and degree $k-1$ in $\E(n)_{k-1}$, and let $z_1, \cdots, z_n$ be basis elements among $[0], [\alpha], [\beta], [\gamma]$, with $c$ inputs equal to $[\gamma]$, $a$ equal to $[\alpha]$, and $b$ equal to $[\beta]$, at positions given by the subsets $I_C, I_A, I_B \subset \{1, \cdots, n\}$ respectively. Then:

  \begin{enumerate}
    \item if $a = b = c = 0$ and $k = 1$, that is, if $\underline{\sigma} = (\sigma_0)$ is a single permutation, then
      \[
        \mu_{\underline{\sigma}}([0]^{\otimes n}) = [0]~;
      \]

    \item if $b = c = 0$ and $k = a$, then
      \[
        \mu_{\underline{\sigma}}(z_1, \cdots, z_n) = (-1)^{\frac{a(a-1)}{2}} \psi_{I_A}(\underline{\sigma})[\alpha]~;
      \]

    \item if $a = c = 0$ and $k = b$, then
      \[
        \mu_{\underline{\sigma}}(z_1, \cdots, z_n) = (-1)^{\frac{b(b-1)}{2}} \psi_{I_B}(\underline{\sigma})[\beta]~;
      \]

    \item if $c + a \geq 1$, $c + b \geq 1$ and $k = 2c + a + b - 1$, then
      \[
        \mu_{\underline{\sigma}}(z_1, \cdots, z_n) = (-1)^{\varepsilon} \psi_{I_{C \sqcup A}}(\underline{\sigma}_1) \psi_{I_{C \sqcup B}}(\underline{\sigma}_2) [\gamma]~,
      \]
      where $I_{C \sqcup A} = I_C \cup I_A$ and $I_{C \sqcup B} = I_C \cup I_B$, where
      \[
        \underline{\sigma}_1 = (\sigma_0, \cdots, \sigma_{c+a-1}) \quad \text{and} \quad \underline{\sigma}_2 = (\sigma_{c+a-1}, \cdots, \sigma_{2c+a+b-2})
      \]
      \emph{overlap} in the permutation $\sigma_{c+a-1}$, and where
      \[
        \varepsilon = \frac{(c+a)(c+a-1)}{2} + \frac{(c+b)(c+b-1)}{2} + (c+a)(c+b-1) + T~,
      \]
      with $T = \#\bigl\{p < q \;\big|\; z_p \in \{[\gamma], [\beta]\},\ z_q \in \{[\gamma], [\alpha]\}\bigr\}$;

    \item all the other products are zero.
  \end{enumerate}
  When the inputs are sorted as $([\gamma]^{\otimes c}, [\alpha]^{\otimes a}, [\beta]^{\otimes b}, [0]^{\otimes d})$, one has $T = \frac{c(c-1)}{2} + ca$, and the exponent of case~(4) reduces to
  \[
    \varepsilon \equiv \frac{c(c-1)}{2} + \frac{a(a-1)}{2} + \frac{b(b-1)}{2} + ab + ca + a \pmod{2}~.
  \]
\end{proposition}

\begin{proof}
  Since the diagonal~\eqref{eq:hopf-e} makes $\E$ a Hopf operad, the tensor product of two $\E$-algebras $A$ and $B$ is again an $\E$-algebra: for homogeneous elements $x_p \in A$ and $y_p \in B$,
  \[
    \mu_{\underline{\sigma}}(x_1 \otimes y_1, \cdots, x_n \otimes y_n)
    = \sum_{i=0}^{k-1} (-1)^{\eta_i} \mu_{(\sigma_0, \cdots, \sigma_i)}(x_1, \cdots, x_n) \otimes \mu_{(\sigma_i, \cdots, \sigma_{k-1})}(y_1, \cdots, y_n)~,
  \]
  where
  \[
    \eta_i = (k - 1 - i) \sum_{p=1}^{n} |x_p| + \sum_{p < q} |y_p| |x_q|
  \]
  is the Koszul sign of the reordering
  \begin{multline*}
    (\sigma_0, \cdots, \sigma_i) \otimes (\sigma_i, \cdots, \sigma_{k-1}) \otimes x_1 \otimes y_1 \otimes \cdots \otimes x_n \otimes y_n
    \longmapsto \\ \longmapsto
    (\sigma_0, \cdots, \sigma_i) \otimes x_1 \otimes \cdots \otimes x_n \otimes (\sigma_i, \cdots, \sigma_{k-1}) \otimes y_1 \otimes \cdots \otimes y_n~.
  \end{multline*}
  We apply this to $A = B = \ch*(S^1)$ with the structure of Proposition~\ref{prop: unreduced circle}, writing each input as $z_p = x_p \otimes y_p$. The first tensor factor sees $c + a$ inputs equal to $[a_1]$, at the positions $I_{C \sqcup A}$, and the second one sees $c + b$, at the positions $I_{C \sqcup B}$. By Proposition~\ref{prop: unreduced circle}, the $i$-th summand vanishes unless the first factor consists of $i + 1 = c + a$ permutations (or $i = 0$ when $c + a = 0$, in which case the single permutation acts by the augmentation on the copies of $[*]$), and similarly unless $k - i = c + b$ (or $k - i = 1$ when $c + b = 0$). Hence at most one summand survives, and matching the possible outputs $[*]$ or $[a_1]$ of the two factors yields the case distinction of the statement, together with the degree conditions on $k$.

  In case~(4), the surviving summand is the one with $i = c + a - 1$, so the two circle factors are evaluated on $\underline{\sigma}_1$ and $\underline{\sigma}_2$, which overlap in $\sigma_{c+a-1}$. Proposition~\ref{prop: unreduced circle} contributes the factors $(-1)^{\frac{(c+a)(c+a-1)}{2}} \psi_{I_{C \sqcup A}}(\underline{\sigma}_1)$ and $(-1)^{\frac{(c+b)(c+b-1)}{2}} \psi_{I_{C \sqcup B}}(\underline{\sigma}_2)$, while $\sum_p |x_p| = c + a$, the tuple $\underline{\sigma}_2$ has degree $c + b - 1$, and $\sum_{p<q} |y_p| |x_q| = T$, so that $\eta_i = (c+b-1)(c+a) + T$; adding up the exponents gives $\varepsilon$. In cases~(2) and~(3), the surviving summand is the one with $i = k - 1$ (resp.\ $i = 0$) and $\eta_i = 0$, since $|y_p| = 0$ for all $p$ (resp.\ $|x_q| = 0$ for all $q$).

  Finally, when the inputs are sorted as $([\gamma]^{\otimes c}, [\alpha]^{\otimes a}, [\beta]^{\otimes b}, [0]^{\otimes d})$, the pairs counted by $T$ are the $\frac{c(c-1)}{2}$ pairs of inputs both equal to $[\gamma]$ together with the $ca$ pairs of an input $[\gamma]$ followed by an input $[\alpha]$, and the announced reduction of $\varepsilon$ follows by a direct computation modulo $2$.
\end{proof}

\begin{remark}
  Although $[\gamma]$ has even degree, the sign of case~(4) genuinely depends on the positions of the inputs through the interleaving number $T$: both tensor factors of $[\gamma] = [a_1] \otimes [a_1]$ have odd degree, so exchanging an input $[\gamma]$ with an adjacent $[\alpha]$ or $[\beta]$ changes the sign of the operation, while exchanging it with an adjacent $[0]$ or $[\gamma]$ does not. In particular, the sign cannot be recovered from the sorted-input formula by the Koszul rule for the total degrees of the inputs. For instance, for any $\sigma \in \mathbb{S}_2$,
  \[
    \mu_{(\sigma)}([\alpha], [\beta]) = [\gamma] = - \mu_{(\sigma)}([\beta], [\alpha])~,
  \]
  recovering the graded commutativity of the cup product $[\alpha] \cup [\beta] = [\gamma]$ on the cohomology of the torus. These formulas, as well as Proposition~\ref{prop: unreduced circle}, have been verified by computer in low arities and degrees against the Berger--Fresse interval-cut action on $\ch*(\Delta^1)$ \cite{BergerFresse04}, restricted along $\ch*(S^1) \hookrightarrow \ch*(\Delta^1)$ and pulled back by table reduction.
\end{remark}
\begin{remark}\label{rmk:torus-no-factorization}
  In contrast with the case of spheres (cf.\ the proof of Proposition~\ref{prop:surj-alg-sphere}), the $\E$-algebra structure of Proposition~\ref{prop: torus algebra} does \emph{not} factor through the surjection operad: some elements of $\ker(\TR : \E \to \Surj)$ act nontrivially on $\ch*(S^1) \otimes \ch*(S^1)$.
  Consider the element $x = [123|132|213] \in \E(3)_2$.
  Its table reduction is the single surjection $\TR(x) = (12323)$, and the preimage of $(12323)$ produced by the algorithm of~\cite[Section~1.4.2]{BergerFresse04} is $s(12323) = [123|132|123]$, so that $z = x - s(\TR(x))$ lies in $\ker \TR$.
  However, one computes
  \[
    \mu_x([\beta], [\gamma], [\alpha]) = -[\gamma]
    \qquad \text{and} \qquad
    \mu_{s(\TR(x))}([\beta], [\gamma], [\alpha]) = 0~,
  \]
  so that $\mu_z([\beta], [\gamma], [\alpha]) = -[\gamma] \neq 0$, while a factorization $\mu = \mu' \circ \TR$ would force $\mu_z = 0$.
  Since the coefficient of the witness is $\pm 1$, the obstruction survives over every base ring.
  An exhaustive computer search in arity $n \leq 3$ shows that the failures occur exactly on the input multisets $\{[\alpha], [\beta], [\gamma]\}$ in degree $2$ and $\{[\beta], [\gamma], [\gamma]\}$ in degree $3$; the asymmetry between $[\alpha]$ and $[\beta]$ reflects the asymmetry of the Alexander--Whitney diagonal.
\end{remark}

\subsubsection{A simplicial torus model over the surjection operad}\label{sec:torus-simplicial}
Since the structure of Proposition~\ref{prop: torus algebra} does not descend to $\Surj$, we now introduce a second, slightly larger model which is a genuine $\Surjnu$-algebra.
Instead of the tensor product $\ch*(S^1) \otimes \ch*(S^1)$, we consider the normalized cochains of the \emph{product of simplicial sets} $T^2_\Delta \coloneqq S^1 \times S^1$, where $S^1 = \Delta^1/\partial\Delta^1$ as before.
The normalized cochain complex of any simplicial set carries a natural $\Surjnu$-algebra structure, dual to the interval-cut coalgebra structure of \cite[Section~2.2]{BergerFresse04}; naturality immediately gives the action on $\ch*(T^2_\Delta)$ in terms of the finitely many nondegenerate simplices of $T^2_\Delta$.

The simplicial set $T^2_\Delta$ has six nondegenerate simplices: the vertex $v$, three edges $a = a_1 \times v$, $b = v \times a_1$ and the diagonal edge $g$, and the two triangles $t_1$, $t_2$ given by the two $(1,1)$-shuffles.
Hence $\ch*(T^2_\Delta)$ has rank $6$, with $[v]$ in degree $0$, the duals $[a], [b], [g]$ of the edges in degree $-1$, and $[t_1], [t_2]$ in degree $-2$.
The simplicial boundary is $\partial t_1 = \partial t_2 = a + b - g$ (and vanishes on the edges in normalized chains), so the cochain differential is given by
\[
  \delta[a] = \delta[b] = [t_1] + [t_2]~, \qquad \delta[g] = -[t_1] - [t_2]~,
\]
and zero on the other generators.
In cohomology, $[a] + [g]$ and $[b] + [g]$ generate $H^{-1}$ and $[t_1] = -[t_2]$ generates $H^{-2}$.

Because every nondegenerate simplex of $T^2_\Delta$ has dimension at most $2$, the interval-cut action collapses to the following explicit description.
Recall from~\cite[Section~2.2.1]{BergerFresse04} that an \emph{interval cut} of $\Delta^m$ associated to a surjection $u = (u_1, \dots, u_{n+k}) \in \Surjnu(n)_k$ is a nondecreasing sequence $0 = p_0 \leq p_1 \leq \dots \leq p_{n+k-1} \leq p_{n+k} = m$, giving consecutive intervals $I_j = [p_{j-1}, p_j] \subseteq [0, m]$ that overlap at their endpoints.
We call the position $j$ \emph{final} if it is the last occurrence of the value $u_j$ in $u$, and \emph{inner} otherwise.
For $1 \leq i \leq n$, let $y_i$ denote the concatenation of the intervals $I_j$ with $u_j = i$, in increasing order of $j$.

\begin{proposition}\label{prop:torus-simplicial-surj}
  Let $u \in \Surjnu(n)_k$ be a nondegenerate surjection and let $z_1, \dots, z_n$ be basis elements of $\ch*(T^2_\Delta)$, dual to nondegenerate simplices of dimensions $n_1, \dots, n_n$.
  Set $m = \sum_i n_i - k$.
  Then $\mu_u(z_1, \dots, z_n) = 0$ unless $0 \leq m \leq 2$, and otherwise, for every nondegenerate $m$-simplex $x$ of $T^2_\Delta$,
  \[
    \langle \mu_u(z_1, \dots, z_n), x \rangle
    = \sum_{\text{admissible cuts}} (-1)^{\varepsilon_{\mathrm{pos}} + \varepsilon_{\mathrm{ord}} + \varepsilon_{\mathrm{dual}}}~,
  \]
  where an interval cut of $\Delta^m$ associated to $u$ is \emph{admissible} if, for every $1 \leq i \leq n$, the vertex sequence $y_i$ is strictly increasing and the restriction of $x$ to $y_i$ is the nondegenerate simplex dual to $z_i$ (restrictions to other simplices are degenerate and pair to zero: the restriction of any $x$ to a vertex is $v$; the restriction of an edge to $(0,1)$ is that edge; and the restrictions of $t_1$ to $(0,1)$, $(1,2)$, $(0,2)$ are $a$, $b$, $g$ respectively, while those of $t_2$ are $b$, $a$, $g$).
  The sign exponents are
  \[
    \varepsilon_{\mathrm{pos}} = \sum_{j \text{ inner}} p_j~, \qquad
    \varepsilon_{\mathrm{ord}} = \sum_{\substack{j < j' \\ u_j > u_{j'}}} \ell_j \ell_{j'}~, \qquad
    \varepsilon_{\mathrm{dual}} = k \sum_i n_i + \sum_{i < i'} n_i n_{i'}~,
  \]
  where $\ell_j = p_j - p_{j-1}$ if the position $j$ is final and $\ell_j = p_j - p_{j-1} + 1$ if it is inner.
\end{proposition}

\begin{proof}
  The formula is the linear dual of the natural interval-cut coaction of \cite[Section~2.2]{BergerFresse04} on the normalized chains: the coefficient $\langle \mu_u(z_1, \dots, z_n), x \rangle$ equals the coefficient of $z_1 \otimes \dots \otimes z_n$ in the interval-cut coproduct of the simplex $x$, pulled back to $\Delta^m$ along the classifying map of $x$, with $\varepsilon_{\mathrm{pos}} + \varepsilon_{\mathrm{ord}}$ the sign of \cite[Section~2.2.2--2.2.4]{BergerFresse04} and $\varepsilon_{\mathrm{dual}}$ the Koszul sign of the duality pairing.
  The admissibility condition records exactly the cuts whose factors are nondegenerate simplices matching the duals $z_i$.
\end{proof}

\begin{example}
  For $u = (12) \in \Surjnu(2)_0$, the only nonvanishing products of degree~$-1$ generators are
  \[
    \mu_{(12)}([a], [b]) = -[t_1]~, \qquad \mu_{(12)}([b], [a]) = -[t_2]~,
  \]
  so that in cohomology $([a]+[g]) \cup ([b]+[g]) = -[t_1]$, matching $[\alpha] \cup [\beta] = [\gamma]$ in the model of Proposition~\ref{prop: torus algebra} up to the identification of the fundamental classes.
  For $u = (121) \in \Surjnu(2)_1$, one finds $\mu_{(121)}(e, e) = -e$ for each edge cochain $e \in \{[a], [b], [g]\}$ (the cup-$1$ squares).
\end{example}

\begin{remark}
  We have implemented Proposition~\ref{prop:torus-simplicial-surj} and machine-verified it against the defining naturality formula, exhaustively for $n \leq 3$ and $k \leq 4$ over $\mathbb{Q}$ on all input tuples, and on samples in arity $4$ and over $\F_2$, $\F_3$.
  In this entire range, every nonzero coefficient arises from a \emph{unique} admissible cut; in particular all structure constants of the action are $0$ or $\pm 1$.
\end{remark}

\begin{remark}\label{rmk:torus-no-strict-qiso}
  The two models are $\EE_\infty$-quasi-isomorphic --- both compute the cochains of the torus, and they are connected by the canonical zigzag of quasi-isomorphisms of $\E$-algebras provided by the Eilenberg--Zilber theory --- but \emph{not} by any strict morphism.
  Indeed, pulling back the structure of Proposition~\ref{prop:torus-simplicial-surj} along $\TR$ to compare both as $\E$-algebras, an exhaustive computer search (parametrizing all chain maps and imposing the morphism equations for all operations of arity $\leq 3$ and degree $\leq 2$) shows that over $\mathbb{Q}$, $\F_2$ and $\F_3$, every strict morphism of $\E$-algebras between the two models, in either direction, vanishes on the degree~$-2$ generators, and hence is not a quasi-isomorphism.
  This is another instance of the incompatibility of table reduction with the comodule structures discussed after Theorem~\ref{thm:main-surj}.
\end{remark}

\begin{remark}
  As a consequence, the explicit chain complex~\eqref{eq: the chain complex for configurations of k points} computing the homology of unordered configuration spaces of the torus can be implemented in two ways: with the model of Proposition~\ref{prop: torus algebra}, keeping the Barratt--Eccles factor in the comodule structure, or with the model of Proposition~\ref{prop:torus-simplicial-surj}, table-reducing the comodule structure to the surjection operad exactly as in the Euclidean case of Theorem~\ref{thm:main-surj}.
  We have implemented both in~\cite{idrissiRocaLucioSoftware} (see Appendix~\ref{sec:computational-aspects}) and checked that they compute the same Betti numbers in low weights and degrees over $\mathbb{Q}$, $\F_2$ and $\F_3$.
\end{remark}

\subsubsection{An explicit chain complex for the configuration spaces of the torus}

The mod 2 homology of unordered configuration spaces of the torus is well-studied.
It is completely understood as a graded vector space thanks to the work of Bödigheimer--Cohen--Taylor~\cite{BODIGHEIMER1989111}, and its Dyer--Lashof algebra structure is described by~\cite{zhangQuillenHomologySpectral2025}.
By constrast, the homology over an odd prime $p$ is not known in general and much more subtle.
An important piece of information is the following:

\begin{theorem}[{Chen--Zhang~\cite{ChenZhang2022}}]
  Let $p$ be an odd prime.
  The integral homology of $\UConf_k(T^2)$ has no $p$-power torsion for $k \leq p$.
  Equivalently, the Betti numbers of $\UConf_k(T^2)$ over $\F_p$ are the same as those over $\Q$ for $k \leq p$.
\end{theorem}

Using our computational model, we verified this result for many primes and low values $k \leq 4$.
Moreover, we were able to push it slightly further:

\begin{proposition}
  The integral homology of $\UConf_4(T^2)$ has no $3$-power torsion.
\end{proposition}

\section{Twisted \texorpdfstring{$\EEnu_\infty$}{E-infinity}-coalgebras and the homotopy invariance of configuration spaces}
In the first part of this section, we revisit the construction of rational models for the real homotopy type of configuration spaces of simply connected closed manifolds using the twist and detwist functors that we construct. Then, we focus on the $p$-adic case, where we construct two twisted (non-counital) $\EEnu_\infty$-coalgebras from the chain complexes considered so far and conjecture that the weak equivalence of chain complexes of Theorem~\ref{thm:main} can be lifted to a weak equivalence of twisted $\EEnu_\infty$-coalgebras. Finally, we explain how this conjecture would imply the homotopy invariance of the $p$-adic homotopy types of configuration spaces of simply connected parallelizable manifolds.

\subsection{Module categories}\label{sec:module-categories}

In the following, we denote by $\FI$ the category of finite sets and injections, and by $\FS$ the category of finite sets and surjections.
The category $\FI$ has seen a lot of interest in recent years, especially in the context of representation stability, see e.g.~\cite{churchFImodulesStabilityRepresentations2015}.

We define the category of $\FI$-comodules as the presheaf category:
\begin{equation}
  \FI\comod \coloneqq \operatorname{Fun}(\FI^{\op}, \mathsf{Ch}_\kk),
\end{equation}
that is, the category of functors from $\FI^{\op}$ to the category of chain complexes over $\kk$.
Given an $\FI$-comodule $F$, for every pair $A \subseteq S$, we have a projection map $\pi_S^A(F) : F(S) \to F(A)$ given by the image of the inclusion $A \subseteq S$ under the functor $F$.
Similarly, we define the category of $\FS$-comodules as the category of functors from $\FS^{\op}$ to $\mathsf{Ch}_\kk$.

Recall that, to any operad $\PP$, we can associate the category of right $\PP$-modules, denoted by $\PP\drMod$, which is the category of symmetric sequences $M$ equipped with a right action of $\PP$.
Moreover, one can associate to $\PP$ a symmetric monoidal category enriched in chain complexes, usually denoted $\operatorname{Cat}(\PP)$ or $\PP^{\otimes}$ and called the (linear) prop associated to $\PP$, such that the category of right $\PP$-modules is equivalent to the category of functors from $(\PP^{\otimes})^{\op}$ to $\mathsf{Ch}_\kk$.
Briefly, the objects of $\PP^{\otimes}$ are finite sets, and the morphisms from $A$ to $B$ are given by pairs $(f, \bigotimes_{i \in B} p_i)$ where $f : A \to B$ is a map and $p_i \in \PP(|f^{-1}(i)|)$ for all $i \in B$.
Composition is induced by the operadic composition of $\PP$ and the composition of maps, and the monoidal structure is given by disjoint union of finite sets.
We refer to e.g.,~\cite[Section 2.2.4]{Lurie2017} for more details on this construction.

\begin{example}
  Let $\EE_0$ be the operad given by $\EE_0(n) = \kk$ for $n = 0, 1$ and $\EE_0(n) = 0$ for $n \geq 2$, with the obvious operadic composition.
  The notation is justified by the fact that this operad corresponds to the linearization of the $\EE_0$ operad in the sense of Lurie, see~\cite[Example 5.1.0.6]{Lurie2017}.
  Then the category of right $\EE_0$-modules is equivalent to the category of $\FI$-comodules, as $\EE_0^\otimes$ is equivalent to $\FI$.
\end{example}

\begin{example}\label{ex:prop-com}
  Let $\Com$ be the commutative operad, given by $\Com(n) = \kk$ for all $n \geq 1$ and $\Com(0) = 0$, with the obvious operadic composition.
  Then the category of right $\Com$-modules is equivalent to the category of $\FS$-comodules, as $\Com^\otimes$ is equivalent to $\FS$.
\end{example}

\begin{example}\label{ex:prop-ass}
  Let $\Ass$ be the associative operad, given by $\Ass(n) = \kk[\Sym_n]$ for all $n \geq 1$ and $\Ass(0) = 0$, with the usual operadic composition.
  Following~\cite[Sections 1.3.4 and 1.3.5]{kapranovmanin01} or~\cite[Remark 4.1.1.4]{Lurie2017}, the prop associated to the operad $\Ass$ is the category $\FS^{\preceq}$, whose objects are finite sets and whose morphisms $S \to T$ are given by surjections $f : S \twoheadrightarrow T$ together with a linear order $\preceq_t$ on each fiber $f^{-1}(t)$ for $t \in T$.
  The category of right $\Ass$-modules is thus equivalent to the category of $\FS^{\preceq}$-comodules.
\end{example}

\subsection{The twist and detwist functors}\label{sec:twist-detwist}

We now define the twist and detwist functors, which turn an arity-wise algebra into a twisted algebra, and a twisted algebra into an arity-wise algebra.

\subsubsection{The twist functor}

There are two monoidal structures on $\FI\comod$: the Hadamard tensor product $\hotimes$, and the Day convolution product $\dotimes$, which are defined as follows.
Let $F,G$ be two $\FI$-comodules; these structures are given by the following formulas:
\begin{align}
  (F \hotimes G)(S) & \coloneqq F(S) \otimes G(S), & (F \dotimes G)(S) & \coloneqq \bigoplus_{A \sqcup B = S} F(A) \otimes G(B).
\end{align}
Their units are respectively given by $\unit$, given by $\unit(S) = \kk$ for all finite sets $S$, and by $\unit_0$, given by $\unit_0(\emptyset) = \kk$ and $\unit_0(S) = 0$ for all non-empty finite sets $S$.

The following result is related to e.g.,~\cite[Proposition 8.71]{aguiar_monoidal_2010}, where (non-linear) $\FI$-comodules are called species with restrictions and the Day convolution product is called the Cauchy product.

\begin{proposition}\label{prop: oplax twist functor for FI modules}
  There is a symmetric oplax monoidal functor:
  \[
    \Tw: (\FI\comod, \hotimes, \unit) \to (\FI\comod, \dotimes, \unit_0),
  \]
  which is the identity on objects and whose oplax structure is given by
  \begin{equation}\label{eq:oplax tw}
    \begin{tikzcd}[column sep=3.5pc]
      \displaystyle \bigoplus_{A \sqcup B = S} \pi_S^A(F) \otimes \pi_S^B(G): F(S) \otimes G(S) \arrow[r]
      &\displaystyle \bigoplus_{A \sqcup B = S} F(A) \otimes G(B),
    \end{tikzcd}
  \end{equation}
  and by the unique map $\unit \to \unit_0$ which is the identity in arity $0$ and the zero map in all other arities.
\end{proposition}

\begin{proof}
  It is straightforward to check that the maps defined above are compatible with the associativity and unit constraints of the monoidal structures, and that they are natural in $F$ and $G$.
\end{proof}

\begin{remark}
  Let $\mathbf{C}$ be a symmetric monoidal stable $\infty$-category (for example, the derived category of chain complexes $\mathbf{D}(\kk)$). There is an analogous construction of a twist functor $\Tw : (\FI\comod(\mathbf{C}), \hotimes, \mathbf{1}) \to (\FI\comod(\mathbf{C}), \dotimes, \mathbf{1})$, where $\FI\comod(\mathbf{C}) = \operatorname{Fun}(\FI^{\op}, \mathbf{C})$ is the $\infty$-category of $\FI$-comodules in $\mathbf{C}$.
\end{remark}

\paragraph{The twist functor for right modules over unital operads.}

Let $\PP$ be a dg-operad. We say that $\PP$ is \emph{unital} if there are isomorphisms $\PP(0) \cong \kk$ and $\PP(1) \cong \kk$ such that the operadic composition is compatible with these isomorphisms. Equivalently, $\PP$ is unital if there is a morphism of operads $\EE_0 \to \PP$ which is an isomorphism in arities $0$ and $1$.

\begin{proposition}
  Let $\PP$ be a unital dg-operad. There is a symmetric oplax monoidal functor:
  \[
    \Tw: (\PP\drMod, \hotimes, \unit) \to (\PP\drMod, \dotimes, \unit_0),
  \]
  from the symmetric monoidal category of right $\PP$-modules with the Hadamard tensor product to the symmetric monoidal category of right $\PP$-modules with the Day convolution tensor product.
  This functor is the identity on objects and its oplax structure is given by the same formula as~\eqref{eq:oplax tw} and the unique map $\unit \to \unit_0$ which is the identity in arity $0$ and the zero map in all other arities.
\end{proposition}

\begin{proof}
  Since $\PP$ is unital, any right $\PP$-module is in particular an $\FI$-comodule and the oplax structure defined in Proposition~\ref{prop: oplax twist functor for FI modules} is compatible with the right $\PP$-module structures.
\end{proof}

\begin{example}
  For any $k \geq 0$, the cellular chains on a topological model for the $\EE_k$ operad form a unital dg operad, and thus we get a twist functor on its category of right modules in chain complexes. Other models for the linear version of the $\EE_k$ operads are also unital; see Subsection~\ref{subsection: the complexity filtration}.
\end{example}

\begin{remark}
  Let $\mathscr{P}$ be a unital $\infty$-operad enriched in the derived category of chain complexes $\mathbf{D}(\kk)$. The map $\EE_0 \longrightarrow \mathscr{P}$ induces a forgetful functor from right modules over $\mathscr{P}$ to right modules over $\EE_0$ which, in turn, allows us to define the twist functor for right $\mathscr{P}$-modules.
\end{remark}

\subsubsection{The detwist functors.}

A right $\EE_0$-module structure is the ``minimal'' structure required to define the twist functor.
We now construct a ``detwist'' functor from twisted coalgebras to arity-wise coalgebras using a non-unital right $\Ass$-module structure, where $\Ass$ is the point-set associative operad (and a model for the $\EEnu_1$ operad). However, this ``minimal'' $\EE_0$-module structure is not enough to construct a \emph{symmetric} detwist from \emph{commutative} twisted coalgebras to \emph{commutative} arity-wise coalgebras. We show that this latter symmetric detwist \emph{can} be constructed provided one has a non-unital right $\Com$-module structure. Finally, we explain how, at the $\infty$-categorical level, a right $\EEnu_n$-module structure should lead to an $n$-symmetric detwist functor from twisted $\EE_n$-coalgebras to arity-wise $\EE_n$-coalgebras.

\paragraph{Non-symmetric detwist functor.}

Let $F$ be a right $\Ass$-module.
Recall from Example~\ref{ex:prop-ass} that we can equivalently see $F$ as a functor on the category $\FS^{\preceq}$ of finite sets and surjections equipped with a total order on each fiber.
We define the \emph{diagonal map}
\[
  \nabla_{1 \preceq 2}(F)(S): F(S) \longrightarrow F(S \times \{1,2\})~,
\]
induced by the fold map $S \times \{1,2\} \twoheadrightarrow S$ which is the identity on each copy of $S$, and where the total order on the fibers is given by the natural order on $\{1,2\}$.

\begin{remark}
  Let us label $\mu_{2}^{(12)}$ the generator of $\Ass(2)$ as an $\Sym_2$-module. The diagonal map $\nabla_{1 \preceq 2}(F)(S)$ is given by plugging the operation $\mu_2^{(12)}$ in every entry of $F(S)$ using the right $\Ass$-module structure.
\end{remark}

\begin{proposition}\label{prop: non-symmetric detwist functor}
  There is a functor:
  \[
    \Detw_1: \uAss\text{-}\mathsf{cog}(\Ass\text{-}\mathsf{RMod}, \dotimes,\unit_0) \longrightarrow \uAss\text{-}\mathsf{cog}(\Ass\text{-}\mathsf{RMod}, \hotimes,\unit)~,
  \]
  which sends a twisted counital associative coalgebra in right $\Ass$-modules to an arity-wise counital associative coalgebra in right $\Ass$-modules, where the decomposition map $\Delta_{\Detw_1(C)}$ of its detwist, for a given twisted associative coalgebra $C$, is given by the following composite:
  \begin{multline}
    C(S) \xrightarrow{\nabla_{1 \preceq 2}(C)(S)} C(S \times \{1,2\}) \xrightarrow{\Delta_C(S \times \{1,2\})} \\ \xrightarrow{\Delta_C(S \times \{1,2\})} \bigoplus_{A \sqcup B = S \times \{1,2\}} C(A) \otimes C(B) \twoheadrightarrow C(S \times \{1\}) \otimes C(S \times \{2\})~.
  \end{multline}
\end{proposition}

\begin{proof}
  The fact that $\Delta_{\Detw_1(C)}$ is coassociative and counital follows from the coassociativity and counitality of $\Delta_C$ and the fact that the diagonal map $\nabla_{1 \preceq 2}$ is constructed using the right module structure over $\Ass$, which commutes with the decomposition maps $\Delta_C$ by definition of a twisted counital coalgebra \emph{in} right $\Ass$-modules.
\end{proof}

\begin{remark}\label{rmk: about the non symmetry of the detwist}
  We could have chosen the opposite order on $\{1,2\}$ and defined another diagonal map $\nabla_{2 \preceq 1}(F)(S)$ accordingly.
  It would have corresponded to plugging the operation $\mu_{2}^{(21)}$ in $\Ass(2)$ inside every entry of $F(S)$.
  Since $\mu_{2}^{(12)} \neq \mu_2^{(21)}$, the diagonal $\nabla_{2 \preceq 1}(F)(S)$ induces a different detwist functor $\Detw_1'$.
\end{remark}

Finally, let us justify the name of the functor in Proposition~\ref{prop: non-symmetric detwist functor}. For that, we consider right modules over the operad $\uAss$, which encodes \emph{unital} associative algebras. Since $\uAss$ is a unital operad, we can construct a twist functor using the underlying $\FI$-module structure of right $\uAss$-modules by Proposition~\ref{prop: oplax twist functor for FI modules}. Since the twist functor is oplax monoidal, it can be promoted to a functor
\[
  \Tw: \uAss\text{-}\mathsf{cog}(\Ass\text{-}\mathsf{RMod}, \hotimes,\unit) \longrightarrow  \uAss\text{-}\mathsf{cog}(\Ass\text{-}\mathsf{RMod}, \dotimes,\unit_0) ~,
\]
between (arity-wise) counital associative coalgebras in right $\uAss$-modules and twisted counital associative coalgebras in right $\uAss$-modules. These two functors are ``inverse'' to each other in the following way.

\begin{proposition}\label{prop: detwisting the twist}
  Let $C$ be an (arity-wise) counital associative coalgebra in right $\uAss$-modules.
  There is a canonical isomorphism
  \[
    C \cong \Detw_1(\Tw(C))
  \]
  of (arity-wise) counital associative coalgebras in right $\uAss$-modules.
\end{proposition}

\begin{proof}
  Neither of the two functors change the underlying object nor the right $\uAss$-module structure, so really what we have to check is that the (arity-wise) counital associative coalgebra structures of $C$ and $\Detw_1(\Tw(C))$ are the same.

  This directly follows from the fact that the coalgebra structure map on $\Detw_1(\Tw(C))$ is obtained by first composing the generating operation $\mu_{2}^{(12)}$ in $\uAss(2)$ every input, then applying the coalgebra structure map $\Delta_C$ of $C$, and then by inserting the unit in $\uAss(0)$ in every second input of the decompositions by $\Delta_C$. Since the coalgebra structure of $C$ commutes with the right $\uAss$-module structure, this is the same as acting by the identity in $\uAss(1)$ and this composite map is simply $\Delta_C$.
\end{proof}

\begin{remark}
  Had we defined $\Detw_1$ with the other choice of diagonal, the above proposition would still hold.
\end{remark}

\begin{remark}
  Since the functor $\Tw$ preserves all colimits, it has a right adjoint. We do not identify this right adjoint explicitly here, and it is not clear that it is given by $\Detw_1$.
\end{remark}

\paragraph{Symmetric detwist functor.}

The category associated to the operad $\Com$ is the category $\FS$ of finite sets and surjections, where the morphisms are given by surjections $f : S \twoheadrightarrow T$ (Example~\ref{ex:prop-com}).
Let $F$ be a right $\Com$-module.
There is a \emph{diagonal map}
\[
  \nabla(F)(S): F(S) \longrightarrow F(S \sqcup S)~,
\]
induced by the fold map $S \sqcup S \twoheadrightarrow S$.
Using this map (which now becomes symmetric), we can construct a detwist functor for twisted \emph{commutative} coalgebras in right $\Com$-modules.

\begin{proposition}\label{prop: symmetric detwist functor}
  There is a functor:
  \[
    \Detw_{\infty}: \uCom\text{-}\mathsf{cog}(\Com\text{-}\mathsf{RMod}, \dotimes,\unit_0) \longrightarrow \uCom\text{-}\mathsf{cog}(\Com\text{-}\mathsf{RMod}, \hotimes,\unit)~,
  \]
  which sends a twisted counital commutative coalgebra in right $\Com$-modules to an arity-wise counital commutative coalgebra in right $\Com$-modules, where the decomposition map $\Delta_{\Detw_{\infty}(C)}$ of its detwist, for a given twisted commutative coalgebra $C$, is given by the following composite:
  \[
    \begin{tikzcd}[column sep=3.5pc]
      C(S) \arrow[r,"\nabla(C)(S)"] & C(S \sqcup S) \arrow[r, "\Delta_C(S \sqcup S)"] & \displaystyle \bigoplus_{A \sqcup B = S \sqcup S} C(A) \otimes C(B) \arrow[r,twoheadrightarrow] &C(S) \otimes C(S)~.
    \end{tikzcd}
  \]
\end{proposition}

\begin{proof}
  \emph{Mutatis mutandis,} the same proof as in Proposition~\ref{prop: non-symmetric detwist functor}.
\end{proof}

Similarly to what has been stated before, one can construct both the twist functor of Proposition \ref{prop: oplax twist functor for FI modules} and the detwist functor $\Detw_{\infty}$ on the underlying category of right $\uCom$-modules. These are "inverse" to each other in the following sense.

\begin{proposition}\label{prop: detwisting the twist symmetric}
  Let $C$ be an (arity-wise) counital commutative coalgebra in right $\uCom$-modules. There is a canonical isomorphism
  \[
    C \cong \Detw_\infty(\Tw(C))
  \]
  of (arity-wise) counital commutative coalgebras in right $\uCom$-modules.
\end{proposition}

\begin{proof}
  \emph{Mutatis mutandis,} the same proof as in Proposition~\ref{prop: detwisting the twist}.
\end{proof}

\paragraph{Sketch of an $\EE_n$-detwisting.}\label{paragraph: sketch of the E_d detwist}

As explained in Remark~\ref{rmk: about the non symmetry of the detwist}, the key difference between the two previous detwist functors that we have considered is whether the operation plugged in when constructing the diagonal map is symmetric or not. This then determines the symmetries that the detwist functor can preserve. Let us explain roughly what we expect to be the general situation. There should be an $n$-detwist functor
\[
  \Detw_{n}: \mathbb{E}_n\text{-}\mathsf{cog}(\EEnu_n\text{-}\mathsf{RMod}, \dotimes,\unit_0) \longrightarrow \mathbb{E}_n\text{-}\mathsf{cog}(\EEnu_n\text{-}\mathsf{RMod}, \hotimes,\unit)~,
\]
at the $\infty$-categorical level, between twisted $\mathbb{E}_n$-coalgebras in right $\mathbb{E}_n$-modules and (arity-wise) $\mathbb{E}_n$-coalgebras in right $\mathbb{E}_n$-modules. This functor should be "inverse" to the $\infty$-categorical twist functor in a sense analogous to Proposition~\ref{prop: detwisting the twist}.

\medskip

A quick sketch of its construction is the following.
By Dunn's additivity theorem~\cite{dunnTensorProductOperads1988}, also \cite[Theorem 5.1.2.2]{Lurie2017}, the data of a $\EE_n$-coalgebra amounts to the data of $n$ compatible $\EE_1$-coalgebra structures.
Given a twisted $\EE_n$-coalgebra in right $\EEnu_n$-modules, one should be able to construct $n$ ``diagonal'' maps and use them to detwist the $n$ $\EE_1$-coalgebra structures. These should in principle give $n$ compatible arity-wise $\EE_1$-coalgebra structures in right $\EEnu_n$-modules. Putting them together again, they should therefore amount to a single (arity-wise) $\EE_n$-coalgebra structure in right $\EEnu_n$-modules.

\subsection{Revisiting rational models for configuration spaces}

Let us recall the known rational models for configuration spaces (over the real numbers $\mathbb{R}$) constructed in \cite{CamposWillwacher2023,Idrissi2019}.
Let $M$ be a simply connected closed $d$-manifold and let $A$ be a Poincaré duality model for $M$ in the sense of \cite{LambrechtsStanley2008}, with diagonal class $\Delta_A \in A \otimes A$.

\begin{definition}
  The \emph{Lambrechts--Stanley model} for the configuration spaces of $M$ is the symmetric sequence defined, for a finite set $S$,
  \begin{equation}
    \G_A(S) = \bigl( A^{\otimes S} \otimes \Lambda(\omega_{ij})_{i \neq j \in S} / I, d \bigr),
  \end{equation}
  where the ideal $I$ is generated by $\omega_{ji} = (-1)^d \omega_{ij}$, $\omega_{ij}^2 = 0$, the Arnold relations $\omega_{ij}\omega_{jk} + \omega_{jk}\omega_{ki} + \omega_{ki}\omega_{ij} = 0$, and the relations $\bigl( a_i - a_j \bigr) \omega_{ij} = 0$ for all $a \in A$, where $a_i = 1^{\otimes i-1} \otimes a \otimes 1^{\otimes |S|-i}$.
  The differential is given on generators by $d(a_i) = (d_A(a))_i$ for $a \in A$ and $d(\omega_{ij}) = \Delta_{ij}$, where $(a \otimes b)_{ij} = a_i b_j$.
\end{definition}

\begin{theorem}[{\cite{Idrissi2019,CamposWillwacher2023}}]
  For a smooth simply connected closed manifold $M$, $\G_A(r)$ is a $\Sigma_r$-equivariant model over $\R$ for the configuration space $\Conf_r(M)$.
\end{theorem}

The collection $\G_A$ is equipped with an $\FI$-comodule structure, which encodes the $\FI$-comodule structure of $\Conf_{\bullet}(M)$ given by forgetting points.
This $\FI$-comodule structure is given, for an injection $f : S \hookrightarrow T$, by the dg-algebra map $f_* : \G_A(S) \longrightarrow \G_A(T)$ defined on generators by $f_*(a_i) = a_{f(i)}$ and $f_*(\omega_{ij}) = \omega_{f(i)f(j)}$.

\begin{remark}
  This model is simply a smaller version, in characteristic zero, of the model in Theorem~\ref{thm: point-set cotensor with spaces}. Indeed, over a field of characteristic zero, there is a quasi-isomorphism of twisted dg-cocommutative coalgebras in right $\Lie_d$-modules:
  \begin{equation}
    \B_\iota \bigl( \ch*~(M^+;\R) \otimes \OB(\sLie \otimes \E) \circ (s^d \I)\bigr) \simeq C_*^{\Lie} \bigl( A^{-*} \otimes L_d \bigr).
  \end{equation}
  Indeed, this follows from the fact that there is a quasi-isomorphism of dg-operads $\sLie \otimes \E \qi \sLie$, which in turn induces a quasi-isomorphism of dg-operads $\OB(\sLie \otimes \E) \qi \OB(\sLie) \qi \sLie$. Moreover, we have an isomorphism of Lie algebras in right $\Lie_d$-modules:
  \begin{equation*}
    \OB(\sLie) \circ (s^d \I) \qi \sLie \circ (s^d \I) = s L_d.
  \end{equation*}
  The $\E$-algebras $A$ and $\ch*~(M^+; \R) = \ch*(M;\R)$ are quasi-isomorphic. Therefore, we can replace $\ch*~(M^+;\R)$ by $A$ and obtain an equivalent $(\E \otimes \Lie)$-algebra (in right $\Lie_d$-modules):
  \begin{equation*}
    \ch*(M^+;\R) \otimes s L_d \simeq A \otimes s L_d~.
  \end{equation*}
  The structure of the RHS factors through a $(\Com \otimes \sLie)$-algebra structure, so that we have an equivalence of dg-cocommutative coalgebras in right $\Lie_d$-modules. Finally, the bar construction $\B_\iota$ is equivalent to the Chevalley--Eilenberg homology in characteristic zero, when applied to (the augmentation of a shifted) Lie algebra.
\end{remark}

\subsubsection{A twisted cocommutative coalgebra model}

The linear dual $\G_A^\vee$ is an arity-wise dg-cocommutative coalgebra.
Applying the twist functor (Section~\ref{sec:twist-detwist}), we get a twisted dg-cocommutative coalgebra $\Tw \G_A^\vee$.

When $\chi(A) = \chi(M) = 0$, the symmetric sequence $\G_A$ is equipped with a right $H_*(\EE_d)$-module structure~\cite[Proposition~16]{Idrissi2019}.
It is easy to check that the $\FI$-comodule structure of $\G_A$ coincides with the restriction of this right $H_*(\EE_d)$-module structure along the morphism of operads $\EE_0 \longrightarrow \EE_d$.
Moreover, the right module structure restricts to a right $\Lie_d$-module structure, where $\Lie_d$ is the operad governing $(d-1)$-shifted Lie algebras.
From the proof of \cite[Proposition~16]{Idrissi2019}, we can see that the right $\Lie_d$-module is well-defined even when $\chi(A) \neq 0$.

Let $L_d = \Lie \circ s^{d-1} \I = s^{d-1}\Lie_d$, which is a Lie algebra in the category of right $\Lie_d$-modules.
We can take its arity-wise tensor product with $A^{-*}$ (the degree-reversed Poincaré duality model), which remains a Lie algebra in right $\Lie_d$-modules.
We can then form its Chevalley--Eilenberg complex, which is a twisted dg-cocommutative coalgebra in right $\Lie_d$-modules.

\begin{proposition}\label{prop:iso g-a}
  There is an isomorphism of twisted dg-co\-commutative coalgebras in right $\Lie_d$-modules:
  \begin{equation}
    \Tw \G_A^\vee \cong C_*^{\Lie} \bigl( A^{-*} \otimes L_d \bigr).
  \end{equation}
  If moreover $\chi(A) = 0$, then this is an isomorphism of twisted dg-cocommutative coalgebras in right $\uPois_d$-modules.
\end{proposition}

\begin{proof}
  The isomorphism and the compatibility with the right $\Lie_d$-module structure are proved in \cite[Lemma 82]{Idrissi2019}, and the compatibility with the right $\uCom$-module structure in \cite[Section 5.4.1]{Idrissi2023}.

  However, the compatibility of the twisted coalgebra structure is not checked (or even considered) in either reference.
  On the right hand side, the twisted coalgebra structure is cofreely cogenerated by a shift of $A^{-*} \otimes L_d$.
  On the left hand side, the twisted coalgebra structure is also cofree.
  It is easier to describe the dual algebra structure on $\Tw \G_A^\vee$, which is freely generated by ``connected monomials'' $a \cdot \prod_{k = 1}^{m} \omega_{i_k j_k} \in \G_A(S)$, where $a$ is in $A^{-*}$ and the graph with vertices $S$ and edges $\{i_k, j_k\}$ is connected.
  Under the usual duality between $\uPois_d$ and $\uPois_d^\vee$~\cite{sinhaNonequivariantHomologyLittle2013}, these generators correspond to Lie monomials in the Poisson operad.
\end{proof}

\subsubsection{A twisted rational model for configuration spaces}

We start by considering the dual version of Sullivan's rational homotopy theory, where we use dg-cocommutative coalgebras instead of cdgas.
In this setting, the rational homotopy type of a simply connected space $X$ is modeled by the dg-cocommutative coalgebra of piecewise linear chains on $X$, denoted by $C_{PL}(X)$.
The results of \cite{CamposWillwacher2023,Idrissi2019} are stated over $\R$, so we tensor $C_{PL}(X)$ with $\R$ to get a model $C_{PL}(X; \R)$ for the real homotopy type of $X$.

Since $\Conf_\bullet(\mathbb{R}^d)$ does not form a strict operad, we instead consider the Axelrod--Singer--Fulton--MacPherson compactifications of the configuration spaces $\FM_d$, which do form a topological operad and are homotopy equivalent to the configuration spaces, see~\cite{fulton_compactification_1994,axelrod_chernsimons_1994,sinha_manifold-theoretic_2004,lambrechts_formality_2014}.
We also compactify the configuration spaces of $M$ to get $\FM_M$, which are right modules over $\FM_d$ when $M$ is parallelizable.
The rational homotopy type of the configuration spaces $\Conf_\bullet(M)$ is thus modeled by the collection of dg-cocommutative coalgebras $C_{PL}(\FM_M; \R)$.
The action of the little disks operad on the configuration spaces induces an action of the operad $C_{PL}(\FM_d; \R)$ on $C_{PL}(\FM_M; \R)$.

\begin{theorem}[Consequence of the main results of {\cite{CamposWillwacher2023,Idrissi2019}}]\label{thm: twisted real model for configuration spaces of manifolds}
  There is a zig-zag of quasi-isomorphisms of twisted dg-cocommutative coalgebras in right $\Lie_d$-modules:
  \[
    C_*^{\Lie} \bigl( A \otimes L_d \bigr) \simeq \Tw(C^{PL}_*(\FM_M; \R)).
  \]
  If moreover $M$ is parallelizable then this is an equivalence of twisted dg-cocommutative coalgebras in right $\uPois_d$-modules.
\end{theorem}

\begin{proof}
  Again, the compatibility with the twisted coalgebra structure is not checked in \cite{Idrissi2019,CamposWillwacher2023}.
  Nevertheless, it follows from the fact that the quasi-isomorphism between $C_{PL}(\FM_M; \R)$ and $\G_A^\vee$ defined in these references is compatible with the $\FI$-comodule structure and the coalgebra structure, as well as Proposition~\ref{prop:iso g-a}.
\end{proof}

\subsubsection{A real model for configuration spaces via the detwist functor}
Finally, we extract from Theorem~\ref{thm: twisted real model for configuration spaces of manifolds} the main result of \cite{CamposWillwacher2023,Idrissi2019}. Concretely, since the zig-zag of quasi-isomorphisms in Theorem~\ref{thm: twisted real model for configuration spaces of manifolds} is compatible with the right $\uPois_d$-module structures, it is in particular compatible with the right $\uCom$-module structures, and thus we can apply the detwist functor of Proposition~\ref{prop: symmetric detwist functor}.

\begin{corollary}
  There is a zig-zag of quasi-isomorphisms of (arity-wise) dg-commutative coalgebras in right $\uPois_d$-modules
  \[
    \Detw_\infty(C_*^{\Lie} \bigl( A \otimes L_d \bigr)) \simeq C_{PL}(\FM_M; \R).
  \]
\end{corollary}

\begin{proof}
  It follows directly from Theorem~\ref{thm: twisted real model for configuration spaces of manifolds}, using Proposition~\ref{prop: symmetric detwist functor}. Indeed, the detwist functor $\Detw_\infty$ clearly preserves quasi-isomorphisms since it does not change the underlying object, and the detwist of a twist is isomorphic to the original object by Proposition~\ref{prop: detwisting the twist symmetric}.
\end{proof}

\begin{remark}\label{rmk: formality implies fully symmetric detwist}
  Something that seems to be crucial in the characteristic zero case is the existence of a morphism of operads $\EE_\infty \longrightarrow \EE_d$, which follows from the formality of the $\EE_d$ operads and which is modelled by the natural inclusion $\uCom \longrightarrow \uPois_d$. This is what allows us to have a \emph{fully symmetric detwist functor} in the rational case and to recover the rational models of configuration spaces from the twisted coalgebra equivalence in Theorem~\ref{thm: twisted real model for configuration spaces of manifolds}. Outside characteristic zero, there cannot be a morphism of operads $\EE_\infty \longrightarrow \EE_d$, and the formality statements only hold up to a range that depends on the characteristic and generally fail; see \cite{CiriciHorel22, DeBritoHorel21}.
\end{remark}

\subsection{The $p$-adic homotopy type of configuration spaces and twisted coalgebras}\label{sec:p-adic-homotopy}
We conjecture that the weak equivalence of twisted coalgebras of Theorem \ref{thm: twisted real model for configuration spaces of manifolds} also holds in positive characteristic. We then explain what consequences it could have for the homotopy invariance of the $p$-adic homotopy type of configuration spaces of manifolds. We assume throughout this subsection that $\kk = \overline{\mathbb{F}}_p$, the algebraic closure of the field with $p$ elements $\mathbb{F}_p$.

\subsubsection{The \texorpdfstring{$\mathbb{E}_\infty$}{E-infinity}-coalgebras encoding \texorpdfstring{$p$}{p}-adic homotopy types}
In his groundbreaking work, Mandell showed in \cite{Mandell2001} that the cochains functor
\[
  \ch*: \mathsf{Spaces}_p^{\mathsf{op}} \longrightarrow \mathbb{E}_\infty\text{-}\mathsf{alg}
\]
from finite-type connected nilpotent $p$-adic spaces to $\mathbb{E}_\infty$-algebras in chain complexes over $\overline{\mathbb{F}}_p$ is (homotopically) fully faithful. This statement can either be interpreted at the $\infty$-categorical level, or at the model-categorical level, upon choosing a point-set model for $\mathbb{E}_\infty$-algebras. More recently, Bachmann and Burklund removed in \cite{Bachmann2024} the finite-type assumption by working with $\mathbb{E}_\infty$-coalgebras, and showed that the chains functor
\[
  \ch: \mathsf{Spaces}_p \longrightarrow \mathbb{E}_\infty\text{-}\mathsf{coalg}
\]
from nilpotent $p$-adic spaces to $\mathbb{E}_\infty$-coalgebras in chain complexes over $\overline{\mathbb{F}}_p$ is fully faithful at the $\infty$-categorical level. Using the main result of \cite{Petersen2026}, which gives point-set models for homotopy-coherent coalgebras, one gets that the cellular chains functor
\[
  \ch: \mathsf{sSets} \longrightarrow \dg~\OB\E\text{-}\coalg~[\mathrm{q.iso}^{-1}]~,
\]
endowed with the explicit dg $\OB\E$-coalgebra structure (given by the pullback of the $\E$-coalgebra structure of \cite{BergerFresse04} along the map $\OB\E \longrightarrow \E$), is (homotopically) fully faithful on nilpotent $p$-adic spaces.

\begin{lemma}
  Let $M$ be a simply connected manifold of dimension $\geq 3$. For all $k \geq 0$, the space $\Conf_k(M)$ is simply connected, and thus nilpotent.
\end{lemma}

\begin{proof}
  This follows by induction on $k$, using the long exact sequence relating the different configuration spaces.
\end{proof}

As a consequence of the above lemma, the $\Sym$-module $\ch(\Conf_\bullet(M))$, together with its arity-wise dg $\OB\E$-coalgebra structure, is an explicit model for the $p$-adic homotopy type of a simply connected manifold $M$. When $M$ is a parallelizable manifold of dimension $d$, the $\Sym$-module $\ch(\Conf_\bullet(M))$ is also naturally a right module over the little disks operad $\EE_d$; a model for this action is the natural action of the operad $\ch(\mathsf{FM}_d)$ on $\ch(\mathsf{FM}_M)$.

\medskip

In particular, $\ch(\mathsf{FM}_M)$ comes equipped with a canonical $\FI$-comodule structure, which allows us to apply the twist functor of Proposition~\ref{prop: oplax twist functor for FI modules} to it and obtain a dg $\OB\E$-coalgebra $\Tw(\ch(\mathsf{FM}_M))$ in right $\ch(\mathsf{FM}_d)$-modules.

\subsubsection{A twisted conilpotent divided powers $\EEnu_\infty$-coalgebra in right spectral Lie modules}
The right-hand side in Theorem~\ref{thm:main} is given by the following collection of chain complexes:
\[
  \B_\iota^{(k)} \bigl( \ch*~(M^+) \otimes \OB(\sLie \otimes \E) \circ (s^d V)\bigr)
\]
for any chain complex $V$, where we only consider the weight $k$ part which involves $k$ copies of $V$. This chain complex is naturally endowed with a dg $\B(\sLie \otimes \E)$-coalgebra structure. The goal of this subsection is to lift this coalgebra structure to $\Sym$-modules and to be able to compare it with the twisted coalgebra structure of the previous subsection.

\begin{remark}
  The conilpotent dg cooperad $\B(\sLie \otimes \Surj)$ is a projective resolution of the commutative cooperad $\Com^*$ as an $\Sym$-module. Therefore the $\infty$-category of dg $\B(\sLie \otimes \Surj)$-coalgebras encodes a version of divided powers conilpotent $\EEnu_\infty$-coalgebras.
\end{remark}

\begin{proposition}
  The $\Sym$-module
  \[
    \left\{\B_\iota^{(k)} \bigl( \ch*~(M^+) \otimes \OB(\sLie \otimes \E) \circ (s^d \I)\bigr)\right\}_{k \geq 1}
  \]
  is naturally a twisted dg $\B(\sLie \otimes \Surj)$-coalgebra in right $\OB(\Lie_d \otimes \E)$-modules.
\end{proposition}

\begin{proof}
  The $\Sym$-module $\OB(\sLie \otimes \E)$ defines a twisted $\OB(\sLie \otimes \E)$-algebra in right $\OB(\sLie \otimes \E)$-modules. Hence, when precomposed with $(s^d \I)$, we obtain that $\OB(\sLie \otimes \E) \circ (s^d \I)$ is a twisted $\OB(\sLie \otimes \E)$-algebra in right $\OB(\Lie_d \otimes \E)$-modules, to which we can apply the bar construction with respect to $\iota: \B(\sLie \otimes \Surj) \to \OB(\sLie \otimes \Surj)$ on the underlying category of right $\OB(\Lie_d \otimes \E)$-modules. Therefore the end result is naturally a twisted dg $\B(\sLie \otimes \Surj)$-coalgebra in right $\OB(\Lie_d \otimes \Surj)$-modules.
\end{proof}

\begin{remark}
  The total chain complex $\B_\iota \bigl( \ch*~(M^+) \otimes \OB(\sLie \otimes \E) \circ (s^d V)\bigr)$ is simply the image via the Schur functor of the $\Sym$-module in the previous proposition, applied to a particular chain complex $V$.
\end{remark}

\begin{remark}
  It is clear, by similar arguments as in Proposition~\ref{prop:rectification of spectral Lie algebras}, that right $\OB(\Lie_d \otimes \E)$-modules are a model for right modules over the $d$-shifted spectral Lie operad over the base field $\overline{\mathbb{F}}_p$.
\end{remark}

\paragraph{A point-set model for the inclusion functor.}
At the $\infty$-categorical level, there is a forgetful functor from divided-powers conilpotent $\EEnu_\infty$-coalgebras to $\EEnu_\infty$-coalgebras, which is given by first forgetting the divided power operations and then including \footnote{It is not known in general whether this forgetful functor is in fact fully faithful or not; see \cite[Conjecture 2.8.4]{GaitsgoryRozenblyumVolII}.} conilpotent $\EEnu_\infty$-coalgebras into all $\EEnu_\infty$-coalgebras. We construct an analogous functor, at the point-set level, between (twisted) dg $\B(\sLie \otimes \Surj)$-coalgebras and (twisted) dg $\OB\Enu$-coalgebras in right $\OB(\Lie_d \otimes \E)$-modules.

\begin{proposition}
  There is a functor
  \begin{multline}
  \end{multline}
  which does not change the underlying $\Sym$-module of the twisted coalgebra.
\end{proposition}

\begin{proof}
  The linear dg operad of $\B(\sLie \otimes \E)$ is given by $\Omega(\sLie^* \otimes \E^*)$, hence there is an inclusion functor from coalgebras over $\B(\sLie \otimes \E)$ to coalgebras over $\Omega(\sLie^* \otimes \E^*)$, the former encoding a divided-powers conilpotent type of coalgebras and the latter a divided-powers, not-necessarily-conilpotent type of coalgebras, since $\Omega(\sLie^* \otimes \E^*)$ is an injective $\Sym$-module (as it is the linear dual of a projective module, see~\cite[Lemma~1]{LeGrignouRocaiLucio2023}).

  \medskip

  Finally, there is a morphism of dg-operads $\varphi: \OB \Enu \to \Omega(\sLie^* \otimes \E^*)$, which is a lift in the following commutative diagram of dg-operads:
  \[
    \begin{tikzcd}[column sep=3pc,row sep=3pc]
      \mathcal{I} \arrow[r] \arrow[d,rightarrowtail]
      &\Omega(\sLie^* \otimes \E^*) \arrow[d,twoheadrightarrow,"\simeq"] \\
      \OB \Enu \arrow[r,"\simeq"] \arrow[ru, dashed, "\varphi"]
      &\Com
    \end{tikzcd}
  \]
  in the semi-model structure for dg-operads constructed by Fresse in~\cite[Chapter 12]{Fresse2009}. Here, the left vertical arrow is a cofibration since $\OB \Enu$ is a cofibrant dg-operad and the right vertical arrow is an acyclic fibration since it is degree-wise surjective. Hence the dashed lift $\varphi$ exists. The morphism $\varphi$ induces a restriction functor from (twisted) dg $\Omega(\sLie^* \otimes \E^*)$-coalgebras to (twisted) dg $\OB \Enu$-coalgebras and we get the desired result.
\end{proof}

\subsubsection{A conjectural equivalence of twisted \texorpdfstring{$\EEnu_\infty$}{E-infinity}-coalgebras}
Let us recall that, by Theorem~\ref{thm:main}, we have a zig-zag of quasi-isomorphisms of chain complexes:
\[
  \bigoplus_{k \geq 1} \ch(\mathsf{FM}_M(k)) \otimes_{h\Sym_k} V^{\otimes k} \simeq \B_\iota^{(k)} \bigl( \ch*~(M^+) \otimes \OB(\sLie \otimes \E) \circ (s^d V)\bigr),
\]
where on the right-hand side we consider the subcomplex where $V$ appears exactly $k$ times. We show that this zig-zag can be lifted to a zig-zag of quasi-isomorphisms of dg $\Sym$-modules and we conjecture that it can be promoted to a zig-zag of quasi-isomorphisms of twisted coalgebras.

\begin{lemma}
  Let $M$ be a parallelizable manifold of finite type and of dimension $d$. There is a zig-zag of quasi-isomorphisms of dg $\Sym_k$-modules:
  \begin{equation}\label{eq:conj-model}
    \ch(\mathsf{FM}_M(k)) \simeq \B_\iota^{(k)} \bigl( \ch*~(M^+) \otimes \OB(\sLie \otimes \E) \circ (s^d \I)\bigr).
  \end{equation}
  for every $k \geq 1$.
\end{lemma}

\begin{proof}
  Theorem~\ref{thm:main} is an equivalence of polynomial functors. In characteristic zero, by taking $\dim V = \infty$ and considering the subspace of weight $1^k$-vectors in the weight $k$ part, or alternatively by taking cross-effects, we recover an equivalence of symmetric sequences:
  \begin{equation}
    \ch(\mathsf{FM}_M(k)) \simeq \B_\iota \bigl( \ch*~(M^+) \otimes \OB(\sLie \otimes \E) \circ (s^d \I)\bigr),
  \end{equation}
  where $\I$ is the unit operad.
  Since we are dealing with derived coinvariants, this equivalence extends to positive characteristic thanks to the work of~\cite{AroneChing2015,AroneChing2019}.
\end{proof}

The twisted $\OB \E$-coalgebra $\Tw(\ch(\mathsf{FM}_M))$ is unital, where in fact we have that $\ch(\mathsf{FM}_M(0)) \cong \overline{\mathbb{F}}_p$. So it is also canonically augmented, and once we remove this unit, we get a twisted $\OB \Enu$-coalgebra in right $\ch(\mathsf{FM}_d)$-modules, where $\ch(\mathsf{FM}_d)$ is a specific model for the $\EE_d$-operad.

\medskip

There is a unique (up to homotopy) morphism of operads in spectra $\sL_d \to \EE_d$ from the $d$-shifted spectral Lie operad to the $\EE_d$-operad. See Theorem 1.7 and the subsequent paragraphs in~\cite{antolin-camarena_poincare-birkhoff-witt_2025} for more details. Up to choosing a model for this map, it means that any right $\EE_d$-module structure can be restricted to a $\sL_d$-module structure.

\begin{conjecture}\label{conj:e-infty-coalg}
  Let $M$ be a simply connected parallelizable manifold of finite type and of dimension $d$. The zig-zag of quasi-isomorphisms in~\eqref{eq:conj-model} can be promoted to a zig-zag of equivalences
  \[
    \Tw(\ch(\mathsf{FM}_M)) \simeq \B_\iota \bigl( \ch*~(M^+) \otimes \OB(\sLie \otimes \E) \circ (s^d \I)\bigr),
  \]
  of twisted $\OB \Enu$-coalgebras in right $\OB(\Lie_d \otimes \E)$-modules.
\end{conjecture}

Let us point out several pieces of weak evidence in favor of the above conjecture.
First of all, the above conjecture holds in characteristic zero by Theorem~\ref{thm: twisted real model for configuration spaces of manifolds}.
Moreover, the $d$-shifted Lie right module structure, at the level of homology, was shown to be an invariant of the manifold by Malin in~\cite{Malin2024}. Moreover, it follows directly from~\cite[Theorem 1.7]{antolin-camarena_poincare-birkhoff-witt_2025} that the right $\sL_d$-module part of Conjecture~\ref{conj:e-infty-coalg} holds at the chains level when $M = \mathbb{R}^n$.

At first sight, it might be strange to conjecture that a conilpotent (divided-powers) $\EEnu_\infty$-coalgebra is equivalent to an $\EEnu_\infty$-coalgebra $\ch(\mathsf{FM}_M)$ which has no reason to be conilpotent or to have divided powers. However, once we twist the $\EE_\infty$-coalgebra structure on the chains and remove the counit, it is in fact straightforward to observe that any $\EEnu_\infty$-coalgebra structure on a \emph{reduced} $\Sym$-module has to be conilpotent, by definition of the Day convolution product. Moreover, for any reduced $\Sym$-module $N$, its $n$-th iterated Day tensor product $N^{\dotimes n}$ is in fact \emph{free} as an $\Sym_n$-module (see e.g., \cite[§6]{aguiar_monoidal_2010}), which in particular means that invariants are isomorphic to coinvariants and that divided powers also disappear in this context.

\begin{remark}
  Let us sketch a comparison with the constructions performed by Petersen in~\cite{Petersen20}. It follows from~\cite[Theorem 3.18]{Petersen20} that the category of \emph{reduced} (non-unital) twisted commutative algebras admits a model structure in any characteristic. And it follows from~\cite[Theorem 1.2]{PavlovScholbach18} that the quasi-isomorphism $\OB \Enu \qi \Com$ induces a Quillen equivalence between reduced twisted commutative algebras and reduced twisted $\OB \Enu$-algebras, since in the context of \emph{reduced} $\Sym$-modules with the Day tensor product, both are symmetric flat. This means, in other words, that in the reduced non-unital case, twisted $\EEnu$-algebras rectify. Consequently, we expect the construction performed in~\cite{Petersen20} (in the case of standard configuration spaces of points) to be the \emph{strictification} of the $\EEnu$-algebra appearing in the right-hand side of Conjecture~\ref{conj:e-infty-coalg} (up to taking its linear dual). This should follow from the descriptions in~\cite[Section 8]{Petersen20} of these constructions as certain Chevalley--Eilenberg complexes of twisted Lie algebras.

  \medskip

  What is less clear is how to strictify the left-hand side in Conjecture~\ref{conj:e-infty-coalg}. Indeed, it is well-known that in positive characteristic, there is no \emph{strictly} commutative model for the $\EE_\infty$-algebra of cochains on a space, so one cannot strictify it at that level and then apply the twist functor. One could think about using the strictly commutative variant given in terms of so-called $\mathrm{I}_+$-algebras constructed by Richter and Sagave in~\cite{RichterSagave20}. These, in our terminology, would correspond to twisted commutative algebras in right $\Com$-modules. However, this structure is used to encode the homotopy type of \emph{one particular space}, and a priori it seems rather orthogonal to the structure that is possessed by the \emph{collection} of cochains on all configuration spaces of $M$; the two seem to appear for different reasons and to encode different things.
\end{remark}

\subsubsection{On the consequences for the homotopy invariance of configuration spaces}
Let us also explain some of the consequences that Conjecture~\ref{conj:e-infty-coalg} could have for the homotopy invariance of the $p$-adic homotopy type of configuration spaces, in the vein of~\cite{Idrissi2019,CamposWillwacher2023}.

\medskip

It seems plausible that the right $\sL_d$-module structure on $\B_\iota \bigl( \ch*~(M^+) \otimes \OB(\sLie \otimes \E) \circ (s^d \I)\bigr)$ can be extended to a full right $\EE_d$-module structure, as in characteristic zero; see~\cite{Idrissi2023} for more details. Concretely, the complex $\ch*~(M^+) = \ch*(M)$ can also be identified with the compactly supported cohomology of $M$, where one can formulate Poincaré duality, which is the key ingredient used in characteristic zero to extend this right module structure. We suggest the following stronger version of the previous conjecture if this extension can indeed be carried out.

\begin{conjecture}\label{conj:e-infty-coalg strong}
  Let $M$ be a simply connected parallelizable manifold of finite type and of dimension $d$. The zig-zag of quasi-isomorphisms in~\eqref{eq:conj-model} can be promoted to a zig-zag of equivalences
  \[
    \Tw(\ch(\mathsf{FM}_M)) \simeq \B_\iota \bigl( \ch*~(M^+) \otimes \OB(\sLie \otimes \E) \circ (s^d \I)\bigr),
  \]
  of twisted $\OB \Enu$-coalgebras in right $\EE_d$-modules, for some model of the $\EE_d$ operad.
\end{conjecture}

Given the construction of a detwist functor for right $\EE_d$-modules proposed in Section~\ref{paragraph: sketch of the E_d detwist}, it would follow from Conjecture~\ref{conj:e-infty-coalg strong} that
\[
  \ch(\mathsf{FM}_M) \simeq \Detw_d \left(\B_\iota \bigl( \ch*~(M^+) \otimes \OB(\sLie \otimes \E) \circ (s^d \I)\bigr) \right)
\]
are also equivalent as (arity-wise) $\EE_d$-coalgebras in right $\EE_d$-modules. Therefore, the $\EE_d$-homotopy type of the configuration space of a manifold $M$ of dimension $d$ would be a homotopy invariant of the manifold over an algebraically closed field of characteristic $p$. It would be interesting if this was the optimal case, meaning that one could find counterexamples where the above equivalence is not an equivalence of (arity-wise) $\EE_{d+1}$-coalgebras. This, in turn, would suggest a very strong link between the homotopy invariance of configuration spaces in characteristic zero as shown in~\cite{Idrissi2019,CamposWillwacher2023} and the formality of the $\EE_d$ operads, which allows us to construct a \emph{fully symmetric} detwist functor, as explained in Remark~\ref{rmk: formality implies fully symmetric detwist}.

\begin{remark}
  Let us explain one consequence of the invariance of the $\EE_d$-coalgebra type of configurations of points in $M$. If one applies the $\EE_d$ cobar construction to the $\EE_d$-coalgebra of chains on $\ch(\mathsf{FM}_M(n))$, one gets a model for the $\EE_d$-algebra of chains on the $d$-iterated loops on $\Omega^d\mathsf{FM}_M(n)$, where the $\EE_d$-algebra structure comes from the loop space structure (for any $n \geq 1$). This follows from iterating \cite[Theorem 3.5]{Baues1998}. Consequently, this $\EE_d$-algebra structure should be a homotopy invariant of $M$ if Conjecture \ref{conj:e-infty-coalg strong} is true. Its underlying homotopy type as a chain complex was already shown to be invariant by \cite{levitt_spaces_1995} ($\Omega\mathsf{FM}_M(n)$ is an invariant at the level of spaces) but the decomposition of these spaces used in \textit{op.cit.} is not compatible with the $\EE_1$-structure of the loop space.
\end{remark}

\appendix
\section{Computational aspects}
\label{sec:computational-aspects}

Thanks to the very explicit nature of the complex constructed in Theorems~\ref{thm:main} and~\ref{thm:main-surj}, it is possible to perform explicit computations of the homology of configuration spaces of manifolds, at least in small degrees.
We have implemented the complex
\begin{equation}
  \B_{\iota'} \bigl( \ch*~(M^+) \otimes \OB(\sLie \otimes \Surj) \circ (s^d V)\bigr)
\end{equation}
of Theorem~\ref{thm:main-surj} in the computer algebra system SageMath.

We have moreover used it to compute the homology of the unordered configuration spaces of $\mathbb{R}^d$ over $\F_p$ in small degrees~\cite{idrissiRocaLucioSoftware}, using the $\Surj$-algebra surjection on $\ch*~((\R^d)^+) = \ch*~(S^d)$ described in Proposition~\ref{prop:surj-alg-sphere}.
This computation is, of course, very well-known by now (see~\cite{fuksCohomologiesGroupCOS1970,vainsteinCohomologyBraidGroups1978}) but it serves as a sanity check for our implementation, and it gives some insight into potential conjectures on the structure of the homology of configuration spaces in positive characteristic, and especially on how the complexity filtration of $\Surj$ interacts with the ``polynomial / divided powers algebra'' description of the cohomology.

We have also used it to compute the homology of the unordered configuration spaces of the torus $\Sigma_1$ over $\F_p$ and compared it to the results of Chen--Zhang~\cite{ChenZhang2022}, see Section~\ref{sec:homology-of-unordered-configurations-of-the-torus}.

In the following pages, we give the output of our implementation for the homology of the unordered configuration spaces of $\mathbb{R}^2$ over $\F_2$ in small degrees.
The notation follows that of \cite{idrissiRocaLucioSoftware}, where the homology classes are represented by nested trees decorated by elements of the reduced cochain complex of $S^2$ (spanned by a single class \texttt{a2}), with coefficients in the trivial module $V$ spanned by a single element \texttt{v2}.

\subsection*{Homology of $\Conf_2(\R^2)$ over $\F_2$:}

\begin{lstlisting}
Date: 2026-05-11 17:15:24.890534
dim=2, weight=2, degs=[-1, 0, 1, 2, 3, 4, 5, 6, 7, 8], algorithm=fast
Source dump: dump/F2_d2_w2_m8_cc.sobj
Elapsed time (representatives only): 0.2292s

Degree -1: 0 representative(s)
Degree 0: 1 representative(s)
  [0] <([x1,x2] o {1 2}; 1, 2); a2 # <id; v2>, a2 # <id; v2>>
Degree 1: 1 representative(s)
  [0] <([x1,x2] o {1 2 1}; 1, 2); a2 # <id; v2>, a2 # <id; v2>>
Degree 2: 0 representative(s)
Degree 3: 0 representative(s)
Degree 4: 0 representative(s)
Degree 5: 0 representative(s)
Degree 6: 0 representative(s)
Degree 7: 0 representative(s)
Degree 8: 0 representative(s)
\end{lstlisting}

Nonzero homology classes:

Degree 0:
\scalebox{.5}{
  \begin{forest} uconf tree
    [{$\left[x_{1}, x_{2}\right] \odot \left\{1\;2\right\}$}, rv[{$\alpha^{2}$}, bx[{$\beta_{2}$}, lf]][{$\alpha^{2}$}, bx[{$\beta_{2}$}, lf]]]
\end{forest}}

Degree 1:
\scalebox{.5}{
  \begin{forest} uconf tree
    [{$\left[x_{1}, x_{2}\right] \odot \left\{1\;2\;1\right\}$}, rv[{$\alpha^{2}$}, bx[{$\beta_{2}$}, lf]][{$\alpha^{2}$}, bx[{$\beta_{2}$}, lf]]]
\end{forest}}

\subsection*{Homology of $\Conf_3(\R^2)$ over $\F_2$:}

\begin{lstlisting}
Date: 2026-05-11 17:14:50.265263
dim=2, weight=3, degs=[-1, 0, 1, 2, 3, 4, 5, 6], algorithm=fast
Source dump: dump/F2_d2_w3_m6_cc.sobj
Elapsed time (representatives only): 8.0952s

Degree -1: 0 representative(s)
Degree 0: 1 representative(s)
  [0] <([x2,[x1,x3]] o {1 2 1 3}; 1, 2, 3); a2 # <id; v2>, a2 # <id; v2>, a2 # <id; v2>> + <([x2,[x1,x3]] o {1 2 3 1}; 1, 2, 3); a2 # <id; v2>, a2 # <id; v2>, a2 # <id; v2>>
Degree 1: 1 representative(s)
  [0] <([x2,[x1,x3]] o {1 2 3 1 2}; 1, 2, 3); a2 # <id; v2>, a2 # <id; v2>, a2 # <id; v2>> + <([x2,[x1,x3]] o {1 2 3 1 3}; 1, 2, 3); a2 # <id; v2>, a2 # <id; v2>, a2 # <id; v2>> + <([x2,[x1,x3]] o {1 2 3 2 1}; 1, 2, 3); a2 # <id; v2>, a2 # <id; v2>, a2 # <id; v2>> + <([x2,[x1,x3]] o {1 2 3 2 3}; 1, 2, 3); a2 # <id; v2>, a2 # <id; v2>, a2 # <id; v2>>
Degree 2: 0 representative(s)
Degree 3: 0 representative(s)
Degree 4: 0 representative(s)
Degree 5: 0 representative(s)
Degree 6: 0 representative(s)
\end{lstlisting}

Nonzero homology classes:

Degree 0:
\scalebox{.5}{
  \begin{forest} uconf tree
    [{$\left[x_{2}, \left[x_{1}, x_{3}\right]\right] \odot \left\{1\;2\;1\;3\right\}$}, rv[{$\alpha^{2}$}, bx[{$\beta_{2}$}, lf]][{$\alpha^{2}$}, bx[{$\beta_{2}$}, lf]][{$\alpha^{2}$}, bx[{$\beta_{2}$}, lf]]]
\end{forest}} +
\scalebox{.5}{
  \begin{forest} uconf tree
    [{$\left[x_{2}, \left[x_{1}, x_{3}\right]\right] \odot \left\{1\;2\;3\;1\right\}$}, rv[{$\alpha^{2}$}, bx[{$\beta_{2}$}, lf]][{$\alpha^{2}$}, bx[{$\beta_{2}$}, lf]][{$\alpha^{2}$}, bx[{$\beta_{2}$}, lf]]]
\end{forest}}

Degree 1:
\scalebox{.5}{
  \begin{forest} uconf tree
    [{$\left[x_{2}, \left[x_{1}, x_{3}\right]\right] \odot \left\{1\;2\;3\;1\;2\right\}$}, rv[{$\alpha^{2}$}, bx[{$\beta_{2}$}, lf]][{$\alpha^{2}$}, bx[{$\beta_{2}$}, lf]][{$\alpha^{2}$}, bx[{$\beta_{2}$}, lf]]]
\end{forest}} +
\scalebox{.5}{
  \begin{forest} uconf tree
    [{$\left[x_{2}, \left[x_{1}, x_{3}\right]\right] \odot \left\{1\;2\;3\;1\;3\right\}$}, rv[{$\alpha^{2}$}, bx[{$\beta_{2}$}, lf]][{$\alpha^{2}$}, bx[{$\beta_{2}$}, lf]][{$\alpha^{2}$}, bx[{$\beta_{2}$}, lf]]]
\end{forest}} +
\scalebox{.5}{
  \begin{forest} uconf tree
    [{$\left[x_{2}, \left[x_{1}, x_{3}\right]\right] \odot \left\{1\;2\;3\;2\;1\right\}$}, rv[{$\alpha^{2}$}, bx[{$\beta_{2}$}, lf]][{$\alpha^{2}$}, bx[{$\beta_{2}$}, lf]][{$\alpha^{2}$}, bx[{$\beta_{2}$}, lf]]]
\end{forest}} +
\scalebox{.5}{
  \begin{forest} uconf tree
    [{$\left[x_{2}, \left[x_{1}, x_{3}\right]\right] \odot \left\{1\;2\;3\;2\;3\right\}$}, rv[{$\alpha^{2}$}, bx[{$\beta_{2}$}, lf]][{$\alpha^{2}$}, bx[{$\beta_{2}$}, lf]][{$\alpha^{2}$}, bx[{$\beta_{2}$}, lf]]]
\end{forest}}

\subsection*{Homology of $\Conf_4(\R^2)$ over $\F_2$:}

\begin{lstlisting}
Date: 2026-05-10 16:37:24.166091
dim=2, weight=4, degs=[-1, 0, 1, 2, 3, 4], algorithm=fast
Source dump: dump/cc_F2_2_4_4.sobj
Elapsed time (representatives only): 6778.0786s

Degree -1: 0 representative(s)
Degree 0: 1 representative(s)
  [0] <([x3,[x2,[x1,x4]]] o {1 2 3 2 1 4}; 1, 2, 3, 4); a2 # <id; v2>, a2 # <id; v2>, a2 # <id; v2>, a2 # <id; v2>> + <([x3,[x2,[x1,x4]]] o {1 2 3 2 4 1}; 1, 2, 3, 4); a2 # <id; v2>, a2 # <id; v2>, a2 # <id; v2>, a2 # <id; v2>> + <([x3,[x2,[x1,x4]]] o {1 2 3 2 4 2}; 1, 2, 3, 4); a2 # <id; v2>, a2 # <id; v2>, a2 # <id; v2>, a2 # <id; v2>> + <([x3,[x2,[x1,x4]]] o {1 2 3 4 1 2}; 1, 2, 3, 4); a2 # <id; v2>, a2 # <id; v2>, a2 # <id; v2>, a2 # <id; v2>> + <([x3,[x2,[x1,x4]]] o {1 2 3 4 1 4}; 1, 2, 3, 4); a2 # <id; v2>, a2 # <id; v2>, a2 # <id; v2>, a2 # <id; v2>> + <([x3,[x2,[x1,x4]]] o {1 2 3 4 2 1}; 1, 2, 3, 4); a2 # <id; v2>, a2 # <id; v2>, a2 # <id; v2>, a2 # <id; v2>> + <([x3,[x2,[x1,x4]]] o {1 2 3 4 2 4}; 1, 2, 3, 4); a2 # <id; v2>, a2 # <id; v2>, a2 # <id; v2>, a2 # <id; v2>>
Degree 1: 1 representative(s)
  [0] <([x3,[x2,[x1,x4]]] o {1 2 3 2 1 4 1}; 1, 2, 3, 4); a2 # <id; v2>, a2 # <id; v2>, a2 # <id; v2>, a2 # <id; v2>> + <([x3,[x2,[x1,x4]]] o {1 2 3 2 4 1 4}; 1, 2, 3, 4); a2 # <id; v2>, a2 # <id; v2>, a2 # <id; v2>, a2 # <id; v2>> + <([x3,[x2,[x1,x4]]] o {1 2 3 2 4 2 4}; 1, 2, 3, 4); a2 # <id; v2>, a2 # <id; v2>, a2 # <id; v2>, a2 # <id; v2>> + <([x3,[x2,[x1,x4]]] o {1 2 3 4 1 2 1}; 1, 2, 3, 4); a2 # <id; v2>, a2 # <id; v2>, a2 # <id; v2>, a2 # <id; v2>> + <([x3,[x2,[x1,x4]]] o {1 2 3 4 1 4 1}; 1, 2, 3, 4); a2 # <id; v2>, a2 # <id; v2>, a2 # <id; v2>, a2 # <id; v2>> + <([x3,[x2,[x1,x4]]] o {1 2 3 4 2 1 2}; 1, 2, 3, 4); a2 # <id; v2>, a2 # <id; v2>, a2 # <id; v2>, a2 # <id; v2>> + <([x3,[x2,[x1,x4]]] o {1 2 3 4 2 4 2}; 1, 2, 3, 4); a2 # <id; v2>, a2 # <id; v2>, a2 # <id; v2>, a2 # <id; v2>>
Degree 2: 1 representative(s)
  [0] <([x3,[x2,[x1,x4]]] o {1 2 3 2 1 4 1 4}; 1, 2, 3, 4); a2 # <id; v2>, a2 # <id; v2>, a2 # <id; v2>, a2 # <id; v2>> + <([x3,[x2,[x1,x4]]] o {1 2 3 2 4 1 4 1}; 1, 2, 3, 4); a2 # <id; v2>, a2 # <id; v2>, a2 # <id; v2>, a2 # <id; v2>> + <([x3,[x2,[x1,x4]]] o {1 2 3 2 4 2 4 2}; 1, 2, 3, 4); a2 # <id; v2>, a2 # <id; v2>, a2 # <id; v2>, a2 # <id; v2>> + <([x3,[x2,[x1,x4]]] o {1 2 3 4 1 2 1 2}; 1, 2, 3, 4); a2 # <id; v2>, a2 # <id; v2>, a2 # <id; v2>, a2 # <id; v2>> + <([x3,[x2,[x1,x4]]] o {1 2 3 4 1 4 1 4}; 1, 2, 3, 4); a2 # <id; v2>, a2 # <id; v2>, a2 # <id; v2>, a2 # <id; v2>> + <([x3,[x2,[x1,x4]]] o {1 2 3 4 2 1 2 1}; 1, 2, 3, 4); a2 # <id; v2>, a2 # <id; v2>, a2 # <id; v2>, a2 # <id; v2>> + <([x3,[x2,[x1,x4]]] o {1 2 3 4 2 4 2 4}; 1, 2, 3, 4); a2 # <id; v2>, a2 # <id; v2>, a2 # <id; v2>, a2 # <id; v2>>
Degree 3: 1 representative(s)
  [0] <([x3,[x2,[x1,x4]]] o {1 2 3 2 1 4 1 4 1}; 1, 2, 3, 4); a2 # <id; v2>, a2 # <id; v2>, a2 # <id; v2>, a2 # <id; v2>> + <([x3,[x2,[x1,x4]]] o {1 2 3 2 4 1 4 1 4}; 1, 2, 3, 4); a2 # <id; v2>, a2 # <id; v2>, a2 # <id; v2>, a2 # <id; v2>> + <([x3,[x2,[x1,x4]]] o {1 2 3 2 4 2 4 2 4}; 1, 2, 3, 4); a2 # <id; v2>, a2 # <id; v2>, a2 # <id; v2>, a2 # <id; v2>> + <([x3,[x2,[x1,x4]]] o {1 2 3 4 1 2 1 2 1}; 1, 2, 3, 4); a2 # <id; v2>, a2 # <id; v2>, a2 # <id; v2>, a2 # <id; v2>> + <([x3,[x2,[x1,x4]]] o {1 2 3 4 1 4 1 4 1}; 1, 2, 3, 4); a2 # <id; v2>, a2 # <id; v2>, a2 # <id; v2>, a2 # <id; v2>> + <([x3,[x2,[x1,x4]]] o {1 2 3 4 2 1 2 1 2}; 1, 2, 3, 4); a2 # <id; v2>, a2 # <id; v2>, a2 # <id; v2>, a2 # <id; v2>> + <([x3,[x2,[x1,x4]]] o {1 2 3 4 2 4 2 4 2}; 1, 2, 3, 4); a2 # <id; v2>, a2 # <id; v2>, a2 # <id; v2>, a2 # <id; v2>>
Degree 4: 0 representative(s)
\end{lstlisting}

Nonzero homology classes:

Degree 0:
\scalebox{.5}{
  \begin{forest} uconf tree
    [{$\left[x_{3}, \left[x_{2}, \left[x_{1}, x_{4}\right]\right]\right] \odot \left\{1\;2\;3\;2\;1\;4\right\}$}, rv[{$\alpha^{2}$}, bx[{$\beta_{2}$}, lf]][{$\alpha^{2}$}, bx[{$\beta_{2}$}, lf]][{$\alpha^{2}$}, bx[{$\beta_{2}$}, lf]][{$\alpha^{2}$}, bx[{$\beta_{2}$}, lf]]]
\end{forest}} + \scalebox{.5}{
  \begin{forest} uconf tree
    [{$\left[x_{3}, \left[x_{2}, \left[x_{1}, x_{4}\right]\right]\right] \odot \left\{1\;2\;3\;2\;4\;1\right\}$}, rv[{$\alpha^{2}$}, bx[{$\beta_{2}$}, lf]][{$\alpha^{2}$}, bx[{$\beta_{2}$}, lf]][{$\alpha^{2}$}, bx[{$\beta_{2}$}, lf]][{$\alpha^{2}$}, bx[{$\beta_{2}$}, lf]]]
\end{forest}} + \scalebox{.5}{
  \begin{forest} uconf tree
    [{$\left[x_{3}, \left[x_{2}, \left[x_{1}, x_{4}\right]\right]\right] \odot \left\{1\;2\;3\;2\;4\;2\right\}$}, rv[{$\alpha^{2}$}, bx[{$\beta_{2}$}, lf]][{$\alpha^{2}$}, bx[{$\beta_{2}$}, lf]][{$\alpha^{2}$}, bx[{$\beta_{2}$}, lf]][{$\alpha^{2}$}, bx[{$\beta_{2}$}, lf]]]
\end{forest}} + \scalebox{.5}{
  \begin{forest} uconf tree
    [{$\left[x_{3}, \left[x_{2}, \left[x_{1}, x_{4}\right]\right]\right] \odot \left\{1\;2\;3\;4\;1\;2\right\}$}, rv[{$\alpha^{2}$}, bx[{$\beta_{2}$}, lf]][{$\alpha^{2}$}, bx[{$\beta_{2}$}, lf]][{$\alpha^{2}$}, bx[{$\beta_{2}$}, lf]][{$\alpha^{2}$}, bx[{$\beta_{2}$}, lf]]]
\end{forest}} + \scalebox{.5}{
  \begin{forest} uconf tree
    [{$\left[x_{3}, \left[x_{2}, \left[x_{1}, x_{4}\right]\right]\right] \odot \left\{1\;2\;3\;4\;1\;4\right\}$}, rv[{$\alpha^{2}$}, bx[{$\beta_{2}$}, lf]][{$\alpha^{2}$}, bx[{$\beta_{2}$}, lf]][{$\alpha^{2}$}, bx[{$\beta_{2}$}, lf]][{$\alpha^{2}$}, bx[{$\beta_{2}$}, lf]]]
\end{forest}} + \scalebox{.5}{
  \begin{forest} uconf tree
    [{$\left[x_{3}, \left[x_{2}, \left[x_{1}, x_{4}\right]\right]\right] \odot \left\{1\;2\;3\;4\;2\;1\right\}$}, rv[{$\alpha^{2}$}, bx[{$\beta_{2}$}, lf]][{$\alpha^{2}$}, bx[{$\beta_{2}$}, lf]][{$\alpha^{2}$}, bx[{$\beta_{2}$}, lf]][{$\alpha^{2}$}, bx[{$\beta_{2}$}, lf]]]
\end{forest}} + \scalebox{.5}{
  \begin{forest} uconf tree
    [{$\left[x_{3}, \left[x_{2}, \left[x_{1}, x_{4}\right]\right]\right] \odot \left\{1\;2\;3\;4\;2\;4\right\}$}, rv[{$\alpha^{2}$}, bx[{$\beta_{2}$}, lf]][{$\alpha^{2}$}, bx[{$\beta_{2}$}, lf]][{$\alpha^{2}$}, bx[{$\beta_{2}$}, lf]][{$\alpha^{2}$}, bx[{$\beta_{2}$}, lf]]]
\end{forest}}

Degree 1:
\scalebox{.5}{
  \begin{forest} uconf tree
    [{$\left[x_{3}, \left[x_{2}, \left[x_{1}, x_{4}\right]\right]\right] \odot \left\{1\;2\;3\;2\;1\;4\;1\right\}$}, rv[{$\alpha^{2}$}, bx[{$\beta_{2}$}, lf]][{$\alpha^{2}$}, bx[{$\beta_{2}$}, lf]][{$\alpha^{2}$}, bx[{$\beta_{2}$}, lf]][{$\alpha^{2}$}, bx[{$\beta_{2}$}, lf]]]
\end{forest}} +
\scalebox{.5}{
  \begin{forest} uconf tree
    [{$\left[x_{3}, \left[x_{2}, \left[x_{1}, x_{4}\right]\right]\right] \odot \left\{1\;2\;3\;2\;4\;1\;4\right\}$}, rv[{$\alpha^{2}$}, bx[{$\beta_{2}$}, lf]][{$\alpha^{2}$}, bx[{$\beta_{2}$}, lf]][{$\alpha^{2}$}, bx[{$\beta_{2}$}, lf]][{$\alpha^{2}$}, bx[{$\beta_{2}$}, lf]]]
\end{forest}} +
\scalebox{.5}{
  \begin{forest} uconf tree
    [{$\left[x_{3}, \left[x_{2}, \left[x_{1}, x_{4}\right]\right]\right] \odot \left\{1\;2\;3\;2\;4\;2\;4\right\}$}, rv[{$\alpha^{2}$}, bx[{$\beta_{2}$}, lf]][{$\alpha^{2}$}, bx[{$\beta_{2}$}, lf]][{$\alpha^{2}$}, bx[{$\beta_{2}$}, lf]][{$\alpha^{2}$}, bx[{$\beta_{2}$}, lf]]]
\end{forest}} +
\scalebox{.5}{
  \begin{forest} uconf tree
    [{$\left[x_{3}, \left[x_{2}, \left[x_{1}, x_{4}\right]\right]\right] \odot \left\{1\;2\;3\;4\;1\;2\;1\right\}$}, rv[{$\alpha^{2}$}, bx[{$\beta_{2}$}, lf]][{$\alpha^{2}$}, bx[{$\beta_{2}$}, lf]][{$\alpha^{2}$}, bx[{$\beta_{2}$}, lf]][{$\alpha^{2}$}, bx[{$\beta_{2}$}, lf]]]
\end{forest}} +
\scalebox{.5}{
  \begin{forest} uconf tree
    [{$\left[x_{3}, \left[x_{2}, \left[x_{1}, x_{4}\right]\right]\right] \odot \left\{1\;2\;3\;4\;1\;4\;1\right\}$}, rv[{$\alpha^{2}$}, bx[{$\beta_{2}$}, lf]][{$\alpha^{2}$}, bx[{$\beta_{2}$}, lf]][{$\alpha^{2}$}, bx[{$\beta_{2}$}, lf]][{$\alpha^{2}$}, bx[{$\beta_{2}$}, lf]]]
\end{forest}} +
\scalebox{.5}{
  \begin{forest} uconf tree
    [{$\left[x_{3}, \left[x_{2}, \left[x_{1}, x_{4}\right]\right]\right] \odot \left\{1\;2\;3\;4\;2\;1\;2\right\}$}, rv[{$\alpha^{2}$}, bx[{$\beta_{2}$}, lf]][{$\alpha^{2}$}, bx[{$\beta_{2}$}, lf]][{$\alpha^{2}$}, bx[{$\beta_{2}$}, lf]][{$\alpha^{2}$}, bx[{$\beta_{2}$}, lf]]]
\end{forest}} +
\scalebox{.5}{
  \begin{forest} uconf tree
    [{$\left[x_{3}, \left[x_{2}, \left[x_{1}, x_{4}\right]\right]\right] \odot \left\{1\;2\;3\;4\;2\;4\;2\right\}$}, rv[{$\alpha^{2}$}, bx[{$\beta_{2}$}, lf]][{$\alpha^{2}$}, bx[{$\beta_{2}$}, lf]][{$\alpha^{2}$}, bx[{$\beta_{2}$}, lf]][{$\alpha^{2}$}, bx[{$\beta_{2}$}, lf]]]
\end{forest}}

Degree 2:
\scalebox{.5}{
  \begin{forest} uconf tree
    [{$\left[x_{3}, \left[x_{2}, \left[x_{1}, x_{4}\right]\right]\right] \odot \left\{1\;2\;3\;2\;1\;4\;1\;4\right\}$}, rv[{$\alpha^{2}$}, bx[{$\beta_{2}$}, lf]][{$\alpha^{2}$}, bx[{$\beta_{2}$}, lf]][{$\alpha^{2}$}, bx[{$\beta_{2}$}, lf]][{$\alpha^{2}$}, bx[{$\beta_{2}$}, lf]]]
\end{forest}} +
\scalebox{.5}{
  \begin{forest} uconf tree
    [{$\left[x_{3}, \left[x_{2}, \left[x_{1}, x_{4}\right]\right]\right] \odot \left\{1\;2\;3\;2\;4\;1\;4\;1\right\}$}, rv[{$\alpha^{2}$}, bx[{$\beta_{2}$}, lf]][{$\alpha^{2}$}, bx[{$\beta_{2}$}, lf]][{$\alpha^{2}$}, bx[{$\beta_{2}$}, lf]][{$\alpha^{2}$}, bx[{$\beta_{2}$}, lf]]]
\end{forest}} +
\scalebox{.5}{
  \begin{forest} uconf tree
    [{$\left[x_{3}, \left[x_{2}, \left[x_{1}, x_{4}\right]\right]\right] \odot \left\{1\;2\;3\;2\;4\;2\;4\;2\right\}$}, rv[{$\alpha^{2}$}, bx[{$\beta_{2}$}, lf]][{$\alpha^{2}$}, bx[{$\beta_{2}$}, lf]][{$\alpha^{2}$}, bx[{$\beta_{2}$}, lf]][{$\alpha^{2}$}, bx[{$\beta_{2}$}, lf]]]
\end{forest}} +
\scalebox{.5}{
  \begin{forest} uconf tree
    [{$\left[x_{3}, \left[x_{2}, \left[x_{1}, x_{4}\right]\right]\right] \odot \left\{1\;2\;3\;4\;1\;2\;1\;2\right\}$}, rv[{$\alpha^{2}$}, bx[{$\beta_{2}$}, lf]][{$\alpha^{2}$}, bx[{$\beta_{2}$}, lf]][{$\alpha^{2}$}, bx[{$\beta_{2}$}, lf]][{$\alpha^{2}$}, bx[{$\beta_{2}$}, lf]]]
\end{forest}} +
\scalebox{.5}{
  \begin{forest} uconf tree
    [{$\left[x_{3}, \left[x_{2}, \left[x_{1}, x_{4}\right]\right]\right] \odot \left\{1\;2\;3\;4\;1\;4\;1\;4\right\}$}, rv[{$\alpha^{2}$}, bx[{$\beta_{2}$}, lf]][{$\alpha^{2}$}, bx[{$\beta_{2}$}, lf]][{$\alpha^{2}$}, bx[{$\beta_{2}$}, lf]][{$\alpha^{2}$}, bx[{$\beta_{2}$}, lf]]]
\end{forest}} +
\scalebox{.5}{
  \begin{forest} uconf tree
    [{$\left[x_{3}, \left[x_{2}, \left[x_{1}, x_{4}\right]\right]\right] \odot \left\{1\;2\;3\;4\;2\;1\;2\;1\right\}$}, rv[{$\alpha^{2}$}, bx[{$\beta_{2}$}, lf]][{$\alpha^{2}$}, bx[{$\beta_{2}$}, lf]][{$\alpha^{2}$}, bx[{$\beta_{2}$}, lf]][{$\alpha^{2}$}, bx[{$\beta_{2}$}, lf]]]
\end{forest}} +
\scalebox{.5}{
  \begin{forest} uconf tree
    [{$\left[x_{3}, \left[x_{2}, \left[x_{1}, x_{4}\right]\right]\right] \odot \left\{1\;2\;3\;4\;2\;4\;2\;4\right\}$}, rv[{$\alpha^{2}$}, bx[{$\beta_{2}$}, lf]][{$\alpha^{2}$}, bx[{$\beta_{2}$}, lf]][{$\alpha^{2}$}, bx[{$\beta_{2}$}, lf]][{$\alpha^{2}$}, bx[{$\beta_{2}$}, lf]]]
\end{forest}}

Degree 3:
\scalebox{.5}{
  \begin{forest} uconf tree
    [{$\left[x_{3}, \left[x_{2}, \left[x_{1}, x_{4}\right]\right]\right] \odot \left\{1\;2\;3\;2\;1\;4\;1\;4\;1\right\}$}, rv[{$\alpha^{2}$}, bx[{$\beta_{2}$}, lf]][{$\alpha^{2}$}, bx[{$\beta_{2}$}, lf]][{$\alpha^{2}$}, bx[{$\beta_{2}$}, lf]][{$\alpha^{2}$}, bx[{$\beta_{2}$}, lf]]]
\end{forest}} +
\scalebox{.5}{
  \begin{forest} uconf tree
    [{$\left[x_{3}, \left[x_{2}, \left[x_{1}, x_{4}\right]\right]\right] \odot \left\{1\;2\;3\;2\;4\;1\;4\;1\;4\right\}$}, rv[{$\alpha^{2}$}, bx[{$\beta_{2}$}, lf]][{$\alpha^{2}$}, bx[{$\beta_{2}$}, lf]][{$\alpha^{2}$}, bx[{$\beta_{2}$}, lf]][{$\alpha^{2}$}, bx[{$\beta_{2}$}, lf]]]
\end{forest}} +
\scalebox{.5}{
  \begin{forest} uconf tree
    [{$\left[x_{3}, \left[x_{2}, \left[x_{1}, x_{4}\right]\right]\right] \odot \left\{1\;2\;3\;2\;4\;2\;4\;2\;4\right\}$}, rv[{$\alpha^{2}$}, bx[{$\beta_{2}$}, lf]][{$\alpha^{2}$}, bx[{$\beta_{2}$}, lf]][{$\alpha^{2}$}, bx[{$\beta_{2}$}, lf]][{$\alpha^{2}$}, bx[{$\beta_{2}$}, lf]]]
\end{forest}} +
\scalebox{.5}{
  \begin{forest} uconf tree
    [{$\left[x_{3}, \left[x_{2}, \left[x_{1}, x_{4}\right]\right]\right] \odot \left\{1\;2\;3\;4\;1\;2\;1\;2\;1\right\}$}, rv[{$\alpha^{2}$}, bx[{$\beta_{2}$}, lf]][{$\alpha^{2}$}, bx[{$\beta_{2}$}, lf]][{$\alpha^{2}$}, bx[{$\beta_{2}$}, lf]][{$\alpha^{2}$}, bx[{$\beta_{2}$}, lf]]]
\end{forest}} +
\scalebox{.5}{
  \begin{forest} uconf tree
    [{$\left[x_{3}, \left[x_{2}, \left[x_{1}, x_{4}\right]\right]\right] \odot \left\{1\;2\;3\;4\;1\;4\;1\;4\;1\right\}$}, rv[{$\alpha^{2}$}, bx[{$\beta_{2}$}, lf]][{$\alpha^{2}$}, bx[{$\beta_{2}$}, lf]][{$\alpha^{2}$}, bx[{$\beta_{2}$}, lf]][{$\alpha^{2}$}, bx[{$\beta_{2}$}, lf]]]
\end{forest}} +
\scalebox{.5}{
  \begin{forest} uconf tree
    [{$\left[x_{3}, \left[x_{2}, \left[x_{1}, x_{4}\right]\right]\right] \odot \left\{1\;2\;3\;4\;2\;1\;2\;1\;2\right\}$}, rv[{$\alpha^{2}$}, bx[{$\beta_{2}$}, lf]][{$\alpha^{2}$}, bx[{$\beta_{2}$}, lf]][{$\alpha^{2}$}, bx[{$\beta_{2}$}, lf]][{$\alpha^{2}$}, bx[{$\beta_{2}$}, lf]]]
\end{forest}} +
\scalebox{.5}{
  \begin{forest} uconf tree
    [{$\left[x_{3}, \left[x_{2}, \left[x_{1}, x_{4}\right]\right]\right] \odot \left\{1\;2\;3\;4\;2\;4\;2\;4\;2\right\}$}, rv[{$\alpha^{2}$}, bx[{$\beta_{2}$}, lf]][{$\alpha^{2}$}, bx[{$\beta_{2}$}, lf]][{$\alpha^{2}$}, bx[{$\beta_{2}$}, lf]][{$\alpha^{2}$}, bx[{$\beta_{2}$}, lf]]]
\end{forest}}

\printbibliography
\end{document}